\documentclass[11pt,a4paper]{article}

\usepackage[a4paper,top=1in,bottom=1in,left=1in,right=1in]{geometry} 
\usepackage{setspace}
\RequirePackage{amsmath,amsfonts,amssymb,amsthm,mathrsfs}
\RequirePackage[numbers,sort&compress]{natbib}
\RequirePackage[colorlinks,citecolor=blue,urlcolor=blue]{hyperref}
\RequirePackage{graphicx}
\usepackage{authblk}
\theoremstyle{plain}

\newtheorem{theorem}{Theorem}[section]
\newtheorem{remark}[theorem]{Remark}
\newtheorem{definition}[theorem]{Definition}
\newtheorem{lemma}[theorem]{Lemma}
\newtheorem{assumption}[theorem]{Assumption}
\newtheorem{proposition}[theorem]{Proposition}
\theoremstyle{remark}

\newcommand{\rIm}{\operatorname{Im}}

\newcommand{\E}{{\mathbb E}}
\newcommand{\tr}{{\rm tr}}

\newcommand{\bx}{{\mathbf x}}
\newcommand{\y}{{\mathbf y}}

\newcommand{\um}{\underline{m}} 
 
\newcommand{\uS}{\underline{S}} 
\newcommand{\uT}{\underline{T}}

\newcommand{\rO}{\mathop{{}\mathrm{O}}\mathopen{}}

\renewcommand{\leq}{\leqslant}
\renewcommand{\geq}{\geqslant}

\newcommand\indep{\protect\mathpalette{\protect\@indep}{\perp}}
\def\@indep#1#2{\mathrel{\rlap{$#1#2$}\mkern2mu{#1#2}}}

\makeatletter \@addtoreset{equation}{section} \makeatother
\hypersetup{
    colorlinks=true,
    linkcolor=blue,
    anchorcolor=blue,
    citecolor=blue,
    urlcolor=blue,
    CJKbookmarks=true
}

\author{Bing-Yi Jing}
\affil{Department of Statistic and Data Science, Southern University of Science and Technology, Shenzhen, China}
\author{Weiming Li}
\affil{School of Statistics and Management, Shanghai University of Finance and Economics, China}
\author{Jiahui Xie}
\affil{Department of Statistics and Data Science,
National University of Singapore, Singapore}
\author{Yangchun Zhang}
\affil{Department of Mathematics, Shanghai University, China}
\author{Wang Zhou}
\affil{Department of Statistics and Data Science,
National University of Singapore, Singapore}

\title{On Convergence Rates of Spiked Eigenvalue Estimates: A General Study of Global and Local Laws in Sample Covariance Matrices}

\begin{document}

\maketitle

\begin{abstract}
This paper investigates global and local laws for sample covariance matrices with general growth rates of dimensions. The sample size $N$ and population dimension $M$ can have the same order in logarithm, which implies that their ratio $M/N$ can approach zero, a constant, or infinity. These theories are utilized to determine the convergence rate of spiked eigenvalue estimates.
\end{abstract}


\tableofcontents

\addtocontents{toc}{\protect\setcounter{tocdepth}{2}}

\section{Introduction}

Covariance matrices serve as fundamental components in multivariate statistics and have versatile applications across fields, gaining heightened significance in high-dimensional data analysis. 
One can refer to \cite{Johnstone01,BS04,bai2008central,2008Operator,bai2009corrections,BSbook,Jiang2013,2014Tracy,bao2015universality,2018A,2019Local,LWYZ2020,zhang2022asymptotic} for an extensive account on statistical applications, \cite{2020High,giglio2022factor} for applications in machine learning and \cite{onatski2009testing,ahn2013eigenvalue,giglio2022factor} in economics, to name a few.

This paper investigates the spectral behaviors of sample covariance matrices with a general correlation structure. The dimension-to-sample size ratio
is allowed to tend towards zero, converge to a positive constant, or diverge to infinity. 
By analyzing these diverse cases, one can understand the spectral evolution of covariance matrices across different asymptotic scenarios, thereby providing comprehensive insights for applications.
To be specific, 
let $T$ be an $M\times M$ deterministic matrix and
\begin{align*}
X=(\mathbf{x}_1,\dots,\mathbf{x}_N)=(x_{ij})_{M,N}
\end{align*}
be a collection of $M\times N$ independent and normalized real or complex variables.
We will consider the following random matrix 
\begin{gather}\label{ass_model}
\mathcal{W}:=TXX^{*}T^{*}.
\end{gather}
This is the classical sample covariance matrix with $\Sigma:=TT^*$ describing population components correlation up to a scaling factor. 
Our main assumptions on this model are listed below.
\begin{itemize}
\item[] 
{\bf Assumption (A1).} The dimensions $M$ and $N$ tend to infinity in such a way that
\begin{gather}\label{ass_M/N}
N\to\infty,\quad M=M_N\to\infty,\quad N^b\lesssim M\lesssim N^a,
\end{gather}
for some positive constants $b\le a$. 

\item[] {\bf Assumption (A2).}
The entries of $X$ satisfy the following moment conditions 
\begin{gather}\label{ass_X}
\mathbb{E}x_{ij}=0,\quad \mathbb{E}|x_{ij}|^2=\frac{1}{\sqrt{MN}},\quad \mathbb{E}|(MN)^{1/4}x_{ij}|^q<C_q.
\end{gather}
for some positive constant $C_q$ and any integer $q\in\mathbb N$.

\item[] {\bf Assumption (A3).}
The empirical spectral distribution $\pi$ of the matrix $\Sigma$
satisfies 
\begin{equation}\label{ass_Sigma}
\pi\left([0,\tau]\right)\leqslant1-\tau\quad\text{ and }\quad \pi\left([0, \tau^{-1}]\right)=1 
\end{equation}
for some small enough constant $\tau>0$. In addition, we assume $\pi$ has finite bulk components.
\end{itemize}

Assumptions (A1)-(A2)-(A3) are commonly employed in Random Matrix Theory. Assumption (A1) illustrates our asymptotic regime, where the dimensions $M$ and $N$ can be logarithmically proportional, i.e., $\log M\sim \log N$. Consequently, their ratio
$$
\phi := \frac MN
$$ 
may have a limit $\phi_\infty$ taking values in the interval $[0,\infty]$. In Assumption (A2), we normalize the entries of $X$ by the factor $(MN)^{1/4}$ and assume the existence of moments of all orders. This standardization allows us to interpret the matrix $\mathcal W$ as a rescaled sample covariance matrix, streamlining our analysis when addressing various growth rates of the ratio $\phi$. Assumption (A3) states that the spectrum of $\Sigma$ is bounded and does not concentrate at zero. A stronger requirement
\begin{gather}\label{ass_T}
T=T^{*}=\Sigma^{1/2}>0
\end{gather}
will be employed to simplify the expressions in technical proofs. Note that \eqref{ass_T} can always be substituted with \eqref{ass_Sigma}.

This paper offers a threefold contribution: establishing global laws for $\mathcal{W}$, deriving local laws for $\mathcal{W}$, and applying these findings to the estimation of spiked eigenvalues. 
These findings are derived through analyzing the Stieltjes transform $m_{\mathcal W}(z)$ of the matrix $\mathcal W$, i.e.,
\begin{align}\label{Stieltjes calW}
    m_{\mathcal W}(z)=\frac1M \tr \mathcal G(z)\quad\text{ with }\quad \mathcal G(z):=\left(\mathcal W-z\right)^{-1}
\end{align}
and  
$$
z=E+\mathrm{i}\eta,\quad \eta>0
$$
denoting the spectral parameter, a complex number in the open upper-half plane $\mathbb{C}_{+}$. The imaginary part $\eta$ of $z$ is called the \emph {spectral resolution}.

\begin{itemize}
\item[]
{\bf Global laws.}
A global law describes the convergence of $m_{\mathcal W}(z)$ with the spectral parameter $z$ independent of $N$ and the spectral resolution $\eta$ of order one.
This type of global law provides tools for determining the limiting spectral distribution of $\mathcal W$ and has been well-established in the regime $\phi\to\phi_\infty\in (0,\infty)$, see \cite{marchenko1967distribution, yin1986limiting,silverstein1995empirical,silverstein1995strong,bhattacharjee2017matrix}.
There are instances where the centralized sample covariance matrix is under consideration,
say, 
$$\mathcal W-\E \mathcal W=TXX^*T^*-\phi^{-\frac12}\Sigma,$$ 
which pulls back the eigenvalues of $\mathcal W$ towards the origin. For references, one can refer to \cite{bai1988convergence,bao2012,chen2012convergence,chen2015clt} when $\phi\to\phi_\infty =0$ or $\infty$.

This paper will first establish global laws for both the sample covariance matrix $\mathcal W$ and its centralized version $\mathcal W-\E \mathcal W$ within the asymptotic framework defined by Assumption (A1). These laws depend on $(M, N)$ but not on a specific limit $\phi_\infty$ and can, therefore, accommodate a broad range of scenarios where $\phi$ approaches $\phi_\infty \in [0,\infty]$. In this way, our findings unify all the previously mentioned results regarding the limiting spectral distribution of $\mathcal W$.

\item[]
{\bf Local laws.}
A local law quantifies the deviation of $m_{\mathcal W}(z)$ from its $(M,N)$-dependent non-random approximate, denoted as $m_0(z)$, for all $z$ with an imaginary part $\eta \gg N^{-1}$. This implies that the local law applies to the spectral parameter $z$, dependent on $N$, allowing the spectral resolution $\eta$ to be significantly smaller than the global scale of 1. The local law is a foundation for establishing a universality theorem, which is similar to the central limit theorem. Universality allows us to determine the asymptotic distribution of eigenvalues of random matrices without imposing strict assumptions on the distribution of entries. For instance, Wigner has observed a physical phenomenon that the eigenvalue gap distribution in a large and complex system can illustrate. This distribution is independent of other intricate structures and solely depends on the symmetry class of the physical system. Essentially, the distribution is universal.
Most studies on the local law for $\mathcal W$ in the literature have primarily focused on scenarios where $\phi$ approaches a positive constant $\phi_\infty\in (0,\infty)$. See \cite{PillaiandYin2014,bao2015universality,knowles2017anisotropic,DingandYang2018,yang2019edge,Wen2021}. As far as we know, the only existing work addressing the general case where $\phi$ can approach zero or infinity is \cite{Bloemendal2014}, established for the case of $\Sigma=I$.

Therefore, our second objective is to establish local laws for the sample covariance matrix $\mathcal W$ with general covariance matrix $\Sigma$ and dimension-to-sample size ratio $\phi\to\phi_\infty \in [0, \infty]$.

\item[]
{\bf Convergence rate of spiked eigenvalue estimates.}

The spiked covariance model for $\Sigma$, originally introduced by \citep{Johnstone01}, illustrates that a small number of eigenvalues of $\Sigma$ are clearly separated from the bulk
and often carry significant information about the population.
Estimating these spikes is an essential statistical inference task, often initiated by the Stieltjes transform $m_{\mathcal W}(z)$. This has been discussed by \cite{mestre2008improved} and \cite{baiandding2012}.

As detailed in the paper, we can derive a more precise bound on the difference between $m_{\mathcal W}(z)$ and its limit by establishing local laws outside the spectrum.  This enables us to determine a convergence rate for the estimation.
\end{itemize}

The remainder of this paper is organized as follows. Section \ref{sec:global} develops global laws for the sample covariance matrix $\mathcal W$. 
Section \ref{sec: local} establishes local laws for $\mathcal W$.
Technical proofs of the results in Section \ref{sec: local} are presented in Sections \ref{main proof} and \ref{Th2&3}. Section \ref{sec: Application} illustrates our application to the estimation of spiked eigenvalues.

\section{Global laws}\label{sec:global}

\subsection{Global laws for sample covariance matrices}
Our first result is on global laws for the sample covariance matrix $\mathcal W$. 
\begin{theorem}\label{mp1}
Suppose that Assumptions (A1)-(A2)-(A3) hold. Then, there exists a deterministic function $m_0(z)$ such that
$$
m_{\mathcal W}(z)-m_0(z)\xrightarrow{a.s.}0,\quad \forall z\in\mathbb C_+.
$$ 
In particular, the function $m=m_0(z)$ is the unique solution to 
\begin{align}\label{eqm2}
    m=\int\frac{1}{x(\phi^{-1/2}-\phi^{1/2}-\phi^{1/2} zm)-z}\pi(\mathrm{d}x)
\end{align}
on the set $\{z: z\in \mathbb C_+, -(1-\phi)/z+\phi m(z)\in \mathbb C_+\}$.
\end{theorem}

\begin{remark}\label{mp1remark}
Theorem \ref{mp1} establishes the strong consistency of the Stieltjes transform 
$m_{\mathcal W}(z)$.
The function $m_0(z)$ is an approximate of the expectation $\E m_{\mathcal W}(z)$ and is uniquely determined by the equation \eqref{eqm2}. An alternative representation of this equation is
\begin{equation}\label{eqm1}
\frac{1}{m_{1}}=-z+\phi^{1/2}\int\frac{x}{1+\phi^{-1/2}m_{1}x}\pi(\mathrm{d}x),
\end{equation}
where $m_1=-(1-\phi)/z+\phi m_0$ approximates the companion Stieltjes transform of $\mathcal W$, see \cite{silverstein1995empirical}.
\end{remark}
\begin{remark}\label{global-re1}

Theorem \ref{mp1} presents a Marc\v{e}nko-Pastur law \citep{marchenko1967distribution} for the eigenvalues of $\mathcal W$, extending the original results to encompass a broader range of the ratio $\phi$. By specializing the limit of this ratio, we can recover several well-established limiting spectral distributions from equations \eqref{eqm2} and \eqref{eqm1}.

\begin{itemize}
\item [I.] The standard case where $\phi\to \phi_\infty\in(0,\infty)$. Equation \eqref{eqm2} converges to that from \citep{marchenko1967distribution} up to a scaling factor $\phi^\frac12$.
\item [II.] The degenerate case where $\phi\to \phi_\infty=0$ with $\Sigma=I_M$.  \cite{bai1988convergence,chen2012convergence} showed that the empirical spectral distribution (ESD) of the centralized sample covariance matrix, i.e.,
$$
\mathcal W-\phi^{-\frac12}I_M,
$$
converges to the standard semicircle law. This can be recovered from \eqref{eqm2} by using the replacement $z\to z+\phi^{-\frac12}$ and then taking the limit as $\phi\to0$, which yields $m_0(z)\to m=m(z)$ satisfying
\begin{align}\label{semicircle-law}
m+\frac{1}{m}+z=0,\quad \forall z\in\mathbb C_+.
\end{align}

\item [III.] The degenerate case where $\phi\to\phi_\infty =\infty$. In this case, the matrix $\mathcal W$ only has $N$ nonzero eigenvalues. It is thus convenient to analyze its companion matrix $W:=X^*\Sigma X$ with the following normalization:
$$
\frac{ W-\phi^{\frac12}a_pI_N}{\sqrt{b_p}}\quad \text{ where } \quad a_p=\frac1M \tr(\Sigma),\quad  b_p=\frac1M \tr(\Sigma^2).
$$
\cite{qiu2023asymptotic} showed that the ESD of this renormalized matrix converges to the standard semicircle law.
To recover this result, we apply the transforms
$$
m_1 \to \frac{m}{\sqrt{b_p}} \quad\text{and}\quad   z\to \sqrt{b_p}\cdot z+\phi^{\frac12}a_p
$$
to \eqref{eqm1}, which gives
\begin{align*}
\frac{\sqrt{b_p}}{m}
&=-\sqrt{b_p}z-\phi^{\frac12}a_p+\phi^{1/2}\int\frac{x\sqrt{b_p}}{\sqrt{b_p}+\phi^{-1/2}mx}\pi(\mathrm{d}x)\\	
&=-\sqrt{b_p}z-\int \frac{x^2 m}{\sqrt{b_p}+\phi^{-1/2}mx}\pi(\mathrm{d}x).
\end{align*}
Then, by taking the limit as $\phi\to\infty$, we get the equation \eqref{semicircle-law} that defines the semicircle law.
\end{itemize}
\end{remark}

To understand the eigenvalue behaviors of the centralized sample covariance matrix $\mathcal W-\E \mathcal W$, we cannot rely on Theorem \ref{mp1} if the covariance matrix $\Sigma$ does not have a spherical shape (i.e., $\Sigma\neq a_pI_M$). Consequently, we introduce a new result to address this problem.

\begin{theorem}\label{mp2}
Suppose that Assumptions (A1)-(A2)-(A3) hold. Let $\tilde m_{\mathcal W}(z)$ be the Stieltjes transform of $\mathcal W-\E \mathcal W$.
Then, there exists a deterministic function $\tilde m_0(z)$ such that
$$
\tilde m_{\mathcal W}(z)-\tilde m_0(z)\xrightarrow{a.s.}0,\quad \forall z\in\mathbb C_+.
$$ 
In particular, the function $m=\tilde m_0(z)$ is the unique solution to 
\begin{align}\label{eqmt}
    \begin{cases}
    &\displaystyle 1+z m=-\frac{g^2}{1+\phi^\frac12 g},\\
    &\displaystyle g=-\int \frac{x}{x g/[1+\phi^\frac12 g]+z} \pi (\mathrm{d}x).
    \end{cases}
\end{align}
on the set $\{z: z\in \mathbb C_+, g\in\mathbb C_+\}$.
\end{theorem}

\begin{remark}
	Theorem \ref{mp2} describes the global eigenvalue distribution of the centralized sample covariance matrix. The auxiliary complex function $g=g(z)$ approximates the following
	 random quantity 
	$$
g_n(z):=\frac1M \tr\left(\mathcal W-\E \mathcal W-z\right)^{-1}\Sigma
	$$
satisfying $g_n(z)-g(z)\to 0$, almost surely. It's evident that if the ratio $\phi\to\phi_\infty= 0$, the system of equations in \eqref{eqmt} reduces to
\begin{align}\label{gsc}
1+zm+g^2=0,\quad  g=\int \frac{x}{-z-xg} \pi(\mathrm{d}x),
\end{align}
which defines a generalized semicircle law; see \cite{baizhang2010,bao2012}. If, in addition,
 $\Sigma= I_M$,
it reduces to \eqref{semicircle-law} that defines the semicircle law.  
\end{remark}

\subsection{Proof of Theorems \ref{mp1} and \ref{mp2}.}

This section is devoted to proving Theorem \ref{mp1} and Theorem \ref{mp2}. Let
\begin{align*}
B=\mathcal W-\theta \E\mathcal W,\quad D(z)=B-zI,\quad m_B(z)=\frac1M \tr D^{-1}(z),  
\end{align*}
where the parameter $\theta \in \{0,1\}$. As demonstrated below, we will establish a general lemma to prove the two theorems.
\begin{lemma}\label{th-all}
Suppose that Assumptions (A1)-(A2)-(A3) hold. 
Then, there exists a deterministic function $m_b(z)$ such that
$$
m_{B}(z)- m_b(z)\xrightarrow{a.s.}0,\quad \forall z\in\mathbb C_+.
$$ 
In particular, the function $m=m_b(z)$ is the unique solution to 
\begin{align}\label{eqmtt}
    \begin{cases}
    &\displaystyle 1+z m=g\phi^{-\frac12}\left[\frac{1}{1+\phi^{1/2} g}-\theta \right],\\
    &\displaystyle g=\int\frac{x}{x\phi^{-1/2}\big[(1+\phi^{1/2}g)^{-1}-\theta \big]-z}\pi(\mathrm{d}x).
    \end{cases}
\end{align}
on the set $\{z: z\in \mathbb C_+, g\in\mathbb C_+\}$.
\end{lemma}
Theorem \ref{mp2} is a direct consequence of this lemma by taking $\theta =1$. To obtain Theorem \ref{mp1}, we set
$$
\theta =0\quad \text{and} \quad 1+\phi^{\frac12}g=b^{-1}, 
$$ 
which gives
\begin{align*}
    m_b=\int\frac{1}{x\phi^{-1/2}b-z}\pi(\mathrm{d}x)=\int\frac{1}{x\phi^{-1/2}(1-\phi-\phi zm_{b})-z}\pi(\mathrm{d}x).
\end{align*}

\noindent
{\bf Proof of Lemma \ref{th-all}.} 
We shall prove this lemma under finite $(4+\delta)$th moments of $\{(MN)^\frac14x_{ij}\}$ for some $\delta>0$.  This proof involves five steps:
\begin{enumerate}
\item For any fixed $z \in \mathbb{C}_{+}, m_B(z)- \E  m_B(z) \rightarrow 0$, a.s.;
\item For any fixed $z \in \mathbb{C}_{+}, \E  m_B(z)-m_b(z) \rightarrow 0$;
\item Except for a null set, $m_B(z)-m_b(z) \rightarrow 0$ for every $z \in \mathbb{C}_{+}$;
\item Uniqueness of the solution to \eqref{eqmtt}.
\end{enumerate}
We will concentrate only on the first two steps, as the third step only involves standard arguments in Random Matrix Theory \citep{BSbook}, and the final step follows a similar procedure for obtaining uniqueness for \eqref{gsc} as described in \cite{baizhang2010}.
Below, we list some notation that will be used throughout the proof. For $j=1,\ldots, N$,
\begin{align*} 
	&\y_j=T\bx_j,\quad D_{j}(z)=D (z)-\y_j\y^*_j,
	\quad \beta_{j}(z)=\frac{1}{1+\y_j^* D_{j}^{-1}(z)\y_j},\nonumber\\ 
	 &b_{j}(z)=\frac{1}{1+\phi^{\frac12}M^{-1}\E \tr\Sigma D_{j}^{-1}(z)},
	\quad 
	\gamma_j(z)=\y_j^* D_j^{-1}(z) \y_j-\phi^\frac12\frac1M \E \tr \Sigma D_j^{-1}(z),\nonumber\\
	&\um_B(z)=-\frac{1-\phi}{z}+\phi m_B(z),\quad
	\um_b(z) =-\frac{1-\phi}{z}+\phi m_b(z),\\
	&b_{N}(z)=\frac{1}{1+\phi^{\frac12}M^{-1}\E \tr\Sigma D^{-1}(z)},\quad V(z)=zI_M- \phi^{-\frac12}(b_{N}(z)-\theta )\Sigma. 
\end{align*} 
We denote by $C$ some constant appearing in inequalities, which may take on different values from one expression to the next.

\noindent
{\bf Step 1. Almost sure convergence of the random part.}
 Let $\E_0(\cdot)$ denote expectation and $\E_j(\cdot)$ denote conditional expectation with respect to the $\sigma$-field generated by $\{\bx_1,\ldots,\bx_j\}$ for $j=1,\ldots,N$. 
Then, by the matrix formula
$$
\left({A}+\boldsymbol{\alpha} \boldsymbol{\beta}^*\right)^{-1}={A}^{-1}-\frac{{A}^{-1} \boldsymbol{\alpha} \boldsymbol{\beta}^*{A}^{-1}}{1+\boldsymbol{\beta}^*{A}^{-1} \boldsymbol{\alpha}},
$$
we can obtain a martingale decomposition of $m_B(z)$ as
\begin{align*} 
m_B(z)-\E  m_B(z) 
&=\frac1M\sum_{j=1}^N(\E_j-\E_{j-1})\tr\left[D ^{-1}(z)-D_{j}^{-1}(z)\right]\nonumber\\
&=\frac1M\sum_{j=1}^N(\E_j-\E_{j-1})\frac{-\y_j^*D_{j}^{-2}(z)\y_j}{1+\y_j^*D_{j}^{-1}(z)\y_j}:= \frac1M\sum_{j=1}^Nd_j(z).
\end{align*} 
For any $z=E+{\rm i}\eta$ with $\eta>0$, we have
\begin{align*}
\left|\frac{\y_j^*D_{j}^{-2}(z)\y_j}{1+\y_j^*D_{j}^{-1}(z)\y_j}\right| \leq 
\frac{\y_j^*\left(D_j^2(E)+\eta^2 I_M\right)^{-1} \y_j}{\rIm\left(1+\y_j^*D_j^{-1}(z)  \y_j\right)}=\frac{1}{\eta}.
\end{align*}
Therefore, $\{d_j\}$ forms a sequence of bounded martingale differences.
In addition, from Lemma \ref{lem: basictool_quadra_concentra},
\begin{align*}
\E\left|\y_j^*D_{j}^{-\ell}(z)\y_j\right|^{2k}
\leq& C|\phi^{\frac12}M^{-1}\tr D_{j}^{-\ell}(z)\Sigma|^{2k}+C\E\left|\y_j^*D_{j}^{-\ell}(z)\y_j-\phi^{\frac12}M^{-1}\tr D_{j}^{-\ell}(z)\Sigma\right|^{2k}\\
 \leq& C\phi^{k}
\end{align*}
for $\ell=1,2$ and any $k\in\mathbb N$. This implies that if $\phi$ tends to zero, the martingale differences $\{d_j\}$ can be bounded by $|\phi|^{\frac12}$ with high probability.
Therefore, by Burkh\"older's inequality, for any $k\geq 1$,  
$$
\begin{aligned}
\E \left|m_B(z)-\E  m_B(z)\right|^{2k} \leq 
\frac{C}{M^{2k}} \E \left(\sum_{j=1}^N\left|d_j(z)\right|^2\right)^{k} \leq \frac{CN^k}{M^{2k}}\max_j \E \left|d_j(z)\right|^{2k}=\mathrm{O}(M^{-k}).
\end{aligned}
$$
This, together with the Borel-Cantelli lemma, implies the almost sure convergence.

{\bf Step 2. Mean convergence.}
For simplicity of notation, we suppress the expression $z$ when it serves as an independent variable of some functions.
Recall the quantities $V$, $b_N$, $\um_B$, and $\um_b$ defined at the beginning of our proof.
We write  
\begin{align*} 
\E m_B-m_b&=\left[\E m_B +\frac1M\tr V^{-1}\right]-\left[\frac1M\tr V^{-1}+m_b\right]
:= S_{N}-T_N\\ 
&=\phi^{-1}\left[\E \um_B +\frac{b_N+\theta  (b_N^{-1}-1)}{z}\right]-\phi^{-1}\left[\frac{b_N+\theta  (b_N^{-1}-1)}{z}+\um_b \right]
:= \uS_N-\uT_N. 
\end{align*} 
We first show that $S_N$ and $\uS_N$ converge to zero. Using the identities 
$$
\y_j^*D^{-1}=\y_j^*D_j^{-1}\beta_j\quad\text{and}\quad \beta_j=b_j-b_j\beta_j\gamma_j,
$$ 
we have 
\begin{align*} 
S_N&=\frac1M\E\tr(D^{-1}+V^{-1})
=\frac1M\E\tr\left[V^{-1}\left(\sum_{j=1}^N\y_j\y_j^*- \phi^{-\frac12}b_{N}\Sigma\right)D^{-1}\right]\nonumber\\ 
&=\frac1M\sum_{j=1}^N\E \beta_j\y_j^*D_j^{-1}V^{-1}\y_j-\phi^{-\frac12}b_{N}\frac1M\E\tr\Sigma D^{-1}V^{-1}\\ 
&=\frac1M\sum_{j=1}^Nb_j \phi^\frac12\E \frac1M\tr D_j^{-1}V^{-1}\Sigma
-\frac1M\sum_{j=1}^N\E b_j\beta_j\gamma_j\y_j^*D_j^{-1}V^{-1}\y_j\\ 
&\quad -\phi^{-\frac12}b_{N}\frac1M\E\tr\Sigma D^{-1}V^{-1}\\
&=\frac{\phi^{-\frac12}}N\sum_{j=1}^N\E\left( \frac{b_j}M\tr D_j^{-1}V^{-1}\Sigma-\frac{b_{N}}M \tr\Sigma D^{-1}V^{-1}\right)
-\frac1M\sum_{j=1}^N\E b_j\beta_j\gamma_j\y_j^*D_j^{-1}V^{-1}\y_j.
\end{align*} 
Since $\max\left\{|\beta_j|, |b_j|, |b_N|\right\}\leq |z|/\eta$ and, for any non-random matrix $A$, 
$$
\frac1M\left|\tr D^{-1}A\right| \leq \frac{\|A\|}{v},\quad \left|\E\tr D_j^{-1}A
-\E\tr D^{-1} A\right| \leq \frac{\phi^\frac12\|A\|}{v},
$$
we get 
\begin{align*}
\left|S_N\right|
&\leq \frac{C}{M}\sum_{j=1}^N\E \left|\gamma_j\y_j^*D_j^{-1}V^{-1}\y_j\right|+\mathrm{o}(1)\\
&\leq \frac{C}{M}\sum_{j=1}^N\E^\frac12 \left|\gamma_j\right|^2\E^\frac12\left|\y_j^*D_j^{-1}V^{-1}\y_j\right|^2+\mathrm{o}(1).
\end{align*}
From Lemma \ref{lem: basictool_quadra_concentra}, we have  
$$
\E|\gamma_j|^2=\E\left|\y_j^* D_j^{-1}(z) \y_j- \phi^\frac12\frac1M \E \tr \Sigma D_j^{-1}(z)\right|^2\leq \frac{C}{N}
$$ and 
$\E|\y_j^*D_j^{-1}V^{-1}\y_j|^2\leq C\phi$.
We thus obtain $S_N\to 0$.
For the term $\uS_n$, using the identities
$$ 
M+z\tr D^{-1}=\tr B D^{-1}
=N-\sum_{j=1}^N\beta_j-\phi^{-\frac12}\theta  \tr \Sigma D^{-1}, 
$$ 
we get
  $
  z\um_B=-N^{-1}\sum_{j=1}^N\beta_j-\phi^{-\frac12}\theta  N^{-1}\tr \Sigma D^{-1}
  $
 and thus 
\begin{align*} 
	\left|\uS_N\right|
	=&\left|
	\frac{N^{-1}\sum_{j=1}^N\E(\beta_j-b_N)}{z\phi}\right|
	\leq C\phi^{-\frac12}
	\frac1N\sum_{j=1}^N\frac1M\E\left|\tr \Sigma D_j^{-1}-\E\tr \Sigma D^{-1}\right|+\mathrm{o}(1)\to0. 
\end{align*}

We next show that $\max\{|T_N|, |\uT_N|\}\to 0$.
Let $1+\phi^{\frac12}g=b^{-1}$. The system \eqref{eqmtt} is equivalent to 
\begin{align*}
    \begin{cases}
    &\displaystyle 1+z m_b=\phi^{-1}[1-b-\theta  (b^{-1}-1)],\\
    &\displaystyle m_b=\int\frac{1}{x\phi^{-1/2}\big[b-\theta \big]-z}\pi(\mathrm{d}x).
    \end{cases}
\end{align*}
Then, we get 
\begin{align}
T_N
&=\int \frac{\pi(\mathrm{d}x)}{z- \phi^{-1/2}(b_{N}-\theta )x}-\int\frac{\pi(\mathrm{d}x)}{z-x\phi^{-1/2}(b -\theta )}\nonumber\\
&=(b_{N}-b)\int \frac{\phi^{-1/2}x \pi(\mathrm{d}x)}{\{z- \phi^{-1/2}(b_{N}-\theta )x\}\{z-x\phi^{-1/2}(b-\theta )\}},\label{tn}\\
\uT_N 
&=\frac{b_N+\theta  (b_N^{-1}-1)}{\phi z}+\frac{\um_b}{\phi} =\frac{b_N-b+\theta  (b_N^{-1}-b^{-1})}{\phi z}. \label{utn}
\end{align} 
With the fact that $\max\{|S_N|, |\uS_N|\}\to 0$, we know that $T_N-\uT_N\to 0$. Therefore, 
\begin{align}
\uT_N-T_N
=&	
\int \frac{\pi(\mathrm{d}x)}{z- \phi^{-1/2}(b_{N}-\theta )x}-\frac{b_N-1+\theta  (b_N^{-1}-1)}{\phi z}-\frac1z \nonumber\\
&-\left\{\int\frac{\pi(\mathrm{d}x)}{z-x\phi^{-1/2}(b -\theta )}-\frac{b-1+\theta  (b^{-1}-1)}{\phi z}-\frac1z\right\}=\mathrm{o}(1).
\label{tntn}
\end{align}
By the uniqueness of the solution to the system of equations in \eqref{eqmtt}, the convergence in \eqref{tntn} implies $b_N-b\to0$. This, together with \eqref{tn} and \eqref{utn}, gives
$$
\max\{|T_N|, |\uT_N|\}\to 0\quad \text{and thus}\quad \E m_B(z)-m_b(z)\to0.
$$
The proof of the lemma is complete.

\section{Local laws}\label{sec: local}
\subsection{Notation and basic tools}
We first prepare some notation. The covariance matrix $\mathcal{W}$ can be regarded as the rescaled typical sample covariance matrix in the sense that $\mathcal{W}=\phi^{-1/2}\mathcal{W}_o$, where
\begin{equation}
    \mathcal{W}_o:=TZZ^{*}T^{*},\quad Z=(\mathbf{z}_1,\dots,\mathbf{z}_N)=(z_{ij})_{M,N},
\end{equation}
and $z_{ij}$'s are independent real or complex random variables satisfying 
\begin{equation}\label{ass_Z}
    \mathbb{E}z_{ij}=0,\quad \mathbb{E}|z_{ij}|^2=\frac{1}{N}\quad \mathbb{E}|N^{1/2}z_{ij}|^q<C_q.
\end{equation}
To simplify the notation, we write
\begin{gather}
\mathcal{W}=\sum_{i=1}^N\mathbf{y}_i\mathbf{y}_i^{*},\quad \mathbf{y}_i:=\Sigma^{1/2}\mathbf{x}_i;\quad
\mathcal{W}_o=\sum_{i=1}^N\mathbf{y}^o_i(\mathbf{y}^o_i)^{*},\quad \mathbf{y}^o_i:=\Sigma^{1/2}\mathbf{z}_i.
\end{gather}
Also, we denote their companions $W$ and $W_o$ as 
\begin{equation*}
W:=X^{*}\Sigma X,\quad W_o:=Z^{*}\Sigma Z.
\end{equation*}
Recall the Green function of $\mathcal W$ that $\mathcal G(z):=\left(\mathcal W-z\right)^{-1}$, analogously, we define the Green function of $W$ as
\begin{align}
G(z):=(W-zI)^{-1},\quad z:=E+\mathrm{i}\eta\in\mathbb{C}_{+},
\end{align}

It is easy to see from the relationship between $W$ and $W_o$ that the typical rates of the eigenvalues of $W$ (or $\mathcal{W}$) fluctuate with the level $\phi^{-1/2}$ of the eigenvalues of $W_o$. Therefore, we conduct our discussion on the level $z=\phi^{-1/2}z_o$ with $z_o:=E_o+\mathrm{i}\eta_o\in\mathbb{C}_{+}$.  Parallelly, we define the Green functions for $W_o,\mathcal{W}_o$ as
\begin{gather}
G_o(z_o):=(W_o-z_oI)^{-1},\quad \mathcal{G}_o(z_o):=(\mathcal{W}_o-z_oI)^{-1},\quad z_o=E_o+\mathrm{i}\eta_o\in\mathbb{C}_{+}.
\end{gather}

Define the empirical spectral distribution (ESD) for $W$ and $\mathcal{W}$ as $\rho_W:=N^{-1}\sum_{i=1}^N\delta_{\lambda_i(W)}$, $\rho_{\mathcal{W}}:=M^{-1}\sum_{i=1}^M\delta_{\lambda_i(\mathcal{W})}$. The Stieltjes transforms of $\mathcal{W}$ \eqref{Stieltjes calW} can also be expressed as  
\begin{align}
m_{\mathcal{W}}:=\int\frac{1}{x-z}\rho_{\mathcal{W}}.
\end{align}
Analogously,
\begin{align}
m_W:=\int\frac{1}{x-z}\rho_W.
\end{align}
Parallelly, we will use the lower index $o$ to denote the corresponding quantities involving $W_o$ (or $\mathcal{W}_o$) in the below, for example, $\rho_{W_o},\rho_{\mathcal{W}_o}$ and $m_{W_o}, m_{\mathcal{W}_o}$. We remark that the level of the variable $z$ will change consequently from our definitions. Therefore, one may observe that 
\begin{gather}
G(z)=\phi^{1/2}G_o(z_o),\quad \mathcal{G}(z)=\phi^{1/2}\mathcal{G}
_o(z_o),\quad m_W=\phi^{1/2}m_{W_o},\quad m_{\mathcal{W}}=\phi^{1/2}m_{\mathcal{W}_o}.
\end{gather}
The following definition is commonly used in the literature, 
\begin{definition}[Stochastic domination] Let
\[A=\left(A^{(n)}(u):n\in\mathbb{N}, u\in U^{(n)}\right),\hskip 10pt B=\left(B^{(n)}(u):n\in\mathbb{N}, u\in U^{(n)}\right),\]
be two families of nonnegative random variables, where $U^{(n)}$ is a possibly $n$-dependent parameter set. We say $A$ is stochastically dominated by $B$, uniformly in $u$, if for any fixed (small) $\epsilon>0$ and (large) $\xi>0$, 
\[\sup_{u\in U^{(n)}}\mathbb{P}\left(A^{(n)}(u)>n^\epsilon B^{(n)}(u)\right)\le n^{-\xi},\]
for large enough $n \ge n_0(\epsilon, \xi)$, and we shall use the notation $A\prec B$ or $A=\mathrm{O}_\prec(B)$. Throughout this paper, the stochastic domination will always be uniform in all parameters that are not explicitly fixed, such as the matrix indices and the spectral parameter $z$.  
\end{definition} 
The following lemmas are useful in our discussion,
\begin{lemma}[Concentration inequality]\label{lem: basictool_quadra_concentra}
Let $A$ be an $M\times M$ matrix with bounded spectral norm, $\mathbf{r}=(r_1,\dots, r_M)^{*}$ where $r_i$'s are independently distributed same as $N^{1/2}z_{11}$ (or $(MN)^{1/4}x_{11}$). Then for any $2\le k\le q/2$,
\[
\mathbb{E}|\mathbf{r}^{*}A\mathbf{r}-\operatorname{tr}A|^k\le C_{2k}(\operatorname{tr}AA^{*})^{k/2}.
\]
\end{lemma}

\begin{lemma}[Resolvent]\label{lem: basictool_resolvent}
For any $\mathcal{T}\subset \{1,2,\cdots,n\}$, we have that 
		\begin{eqnarray*}
			G_{ii}^{(\mathcal{T})}(z) & = & -\frac{1}{z+z\mathbf{y}_{i}^{*}{\cal G}^{(i\mathcal{T})}(z)\mathbf{y}_{i}},\qquad\forall i\in\mathcal{I}\setminus \mathcal{T},\\
			G_{ij}^{(\mathcal{T})}(z) & = & zG_{ii}^{(\mathcal{T})}(z)G_{jj}^{(i\mathcal{T})}(z)\mathbf{y}_{i}^{*}{\cal G}^{(ij\mathcal{T})}(z)\mathbf{y}_{j},\qquad\forall i,j\in\mathcal{I}\setminus \mathcal{T},i\ne j,\\
			G_{ij}^{(\mathcal{T})}(z) & = & G_{ij}^{(k\mathcal{T})}(z)+\frac{G_{ik}^{(\mathcal{T})}(z)G_{kj}^{(\mathcal{T})}(z)}{G_{kk}^{(\mathcal{T})}(z)},\qquad\forall i,j,k\in\mathcal{I}\setminus \mathcal{T},i,j\ne k.
		\end{eqnarray*}
Moreover,
\begin{gather*}
      \sum_{1\le i\le N}|G_{ji}|^2=\sum_{1\le i\le N}|G_{ij}|^2=\frac{\operatorname{Im}G_{jj}}{\eta},\ 
    \|\mathcal{G}\|_F^2=\frac{\operatorname{Im}\operatorname{tr}(\mathcal{G})}{\eta}.
\end{gather*} 		
\end{lemma}

\begin{lemma}[Large deviation bounds]\label{lem: baisctool_largedevi}
    Let $\mathbf{r}_i,\mathbf{r}_j,i\neq j$ be two independent random vectors from either columns of the matrix $X$ satisfying \eqref{ass_X} or columns of the matrix $Z$ satisfying \eqref{ass_Z}. Suppose $A$ is  an $M\times M$ matrix and $\mathbf{b}$
		an $M$-dimensional vector, where $A$ and $\mathbf{b}$ maybe
		complex-valued and independent of $\mathbf{r}_i,\mathbf{r}_j$. Then as $M\rightarrow\infty$,
  \begin{itemize}
  \item [(i).] If $\mathbf{r}_i$'s come from the columns of $X$, then
  \begin{gather}
      |\mathbf{b}^{*}\mathbf{x}_i|\prec\big(\frac{\|\mathbf{b}^2\|}{\sqrt{MN}}\big)^{1/2},\\
      |\mathbf{x}_i^{*}A\mathbf{x}_i-\frac{1}{\sqrt{MN}}\operatorname{tr}A|\prec\frac{1}{\sqrt{MN}}\|A\|_F,\\
      |\mathbf{x}_i^{*}A\mathbf{x}_j|\prec\frac{1}{\sqrt{MN}}\|A\|_F.
  \end{gather}
\item [(ii).] If $\mathbf{r}_i$'s come from the columns of $Z$, then
\begin{gather}
      |\mathbf{b}^{*}\mathbf{z}_i|\prec\big(\frac{\|\mathbf{b}^2\|}{N}\big)^{1/2},\\
      |\mathbf{z}_i^{*}A\mathbf{z}_i-\frac{1}{N}\operatorname{tr}A|\prec\frac{1}{N}\|A\|_F,\\
      |\mathbf{z}_i^{*}A\mathbf{z}_j|\prec\frac{1}{N}\|A\|_F.
  \end{gather}
  \end{itemize}
\end{lemma}

\begin{lemma}[Interlacing bounds]\label{lem: basictool_diffbounds}
The following estimates hold uniformly for $z\in\mathbb{C}_{+}$ and $C>0$,
\begin{gather}
    \|G\|+\|\mathcal{G}\|\le\frac{C}{\eta}\\
    |{\rm tr}(G^{(i)}-G)|  \le  \eta^{-1},\\
	|{\rm tr}({\cal G}^{(i)}-{\cal G})| \le |z|^{-1}+\eta^{-1},\\
    |\operatorname{Im}{\rm tr}({\cal G}^{(i)}-{\cal G})| \le \eta|z|^{-2}+\eta^{-1},
\end{gather}
and in particular for any $\mathcal{T}\subset\{1,\dots,M\}$, we have 
\begin{gather}
    |m_W-m_W^{(\mathcal{T})}|\le\frac{|\mathcal{T}|}{N\eta}.
\end{gather}
\end{lemma}
\begin{proof}
Lemma \ref{lem: basictool_quadra_concentra} follows lines in the proof of \cite[Lemma 2.7]{silverstein1995empirical}. The proof of Lemma \ref{lem: basictool_resolvent} can be found in \cite{DingandYang2018,PillaiandYin2014,yang2019edge}. And Lemma \ref{lem: baisctool_largedevi} follows assumption \eqref{ass_X} whose proof can be found in \cite{bao2015universality}. Finally, lines around \cite[Lemma 5.4]{Wen2021} can be used to show Lemma \ref{lem: basictool_diffbounds}. We omit further details here.
\end{proof}

\subsection{Main result}
Recall the global law in Theorem \ref{mp1} and Remark \ref{mp1remark}. Since when $\phi$ tends to zero or infinity several terms in these equations will be pretty large, it is convenient to use
\begin{gather}
    \frac{1}{m_{1o}}=-z_o+\phi\int\frac{x}{1+m_{1o}x}\pi(\mathrm{d}x),
\end{gather}
where 
\begin{align}\label{m1ozo}
    m_{1o}(z_o)=\phi^{-1/2}m_{1}(z).
\end{align}
By the results in \cite{knowles2017anisotropic}, $m_{1o}$ can be characterized as the unique solution of the equation
\[
\widetilde{z}=f(m),\quad \operatorname{Im}m>0,\quad \widetilde{z}\in\mathbb{C}_{+},
\]
where we defined
\begin{gather}
    f(x):=-\frac{1}{x}+\phi\sum_{i=1}^M\frac{\pi(\{\sigma_i\})}{x+\sigma_i^{-1}}.
\end{gather}

For any fixed $M,N$, $f(x)$ is smooth on the $M+1$ open intervals of $\mathbb{R}$ defined through
\[
I_1:=(-\sigma_1^{-1},0),\quad I_i:=(-\sigma_i^{-1},-\sigma_{i-1}^{-1}),\quad I_0:=\mathbb{R}\setminus\bigcup_{i=1}^M \bar{I}_i,
\]
where $\bar{I}_i$ is the closed cover of $I_i$. We note that those intervals $I_i$'s can be duplicated. As in \cite{knowles2017anisotropic}, we introduce the multiset  $\mathcal{C}\subset\bar{\mathbb{R}}$ of critical points of $f$, where $\bar{\mathbb{R}}:=\mathbb{R}\cup\{\infty\}$. By \cite[Lemma 2.4]{knowles2017anisotropic}, one may observe that $|\mathcal{C}\cap I_0|=|\mathcal{C}\cap I_1|=1$ and $\mathcal{C}\cap I_i\in\{0,2\}$ for $i=2,\dots,M$. From which we may deduce $|\mathcal{C}|=2p$ is even for some integer $p\le M/2$. We denote $x_1\ge x_2\ge\cdots\ge x_{2p-1}$ be the $2p-1$ critical points in $I_1\cup I_2\cup\cdots\cup I_M$ and $x_{2p}$ be the unique critical point in $I_0$. The following result gives the behavior of $f$ at each critical point as in \cite{knowles2017anisotropic}.
\begin{lemma}
For any fixed $M,N$, there are $2p$ critical points $x_1\ge x_2\ge\dots\ge x_{2p}$ with the critical values $a_1\ge a_2\ge\dots\ge a_{2p}$ such that $a_k=f(x_k)$ and $x_k=m_{1o}(a_k)$, where the critical points satisfy:
\begin{gather}
    f^{\prime}(x_k)=0,\quad f^{\prime\prime}(x_k)\sim 1,\quad f^{(3)}(x_k)\lesssim C.
\end{gather}
\end{lemma}
\begin{remark}\label{rem: regular edges}
We remark that the above result holds for fixed $M,N$. For the case where $M,N$ are sufficiently large and $M\gg N$, $\phi$ diverges and there will be two critical points $x_1\in(-\sigma_1^{-1},0)$ and $x_2>0$ such that $f^{\prime}(x_{1,2})=0$. Then we have
\begin{gather*}
x_1\asymp -\phi^{-1/2},\; f^{\prime\prime}(x_1)\asymp \phi^{3/2},\quad x_2\asymp \phi^{-1/2},\; f^{\prime\prime}(x_2)\asymp -\phi^{3/2}.
\end{gather*}
Therefore the bulk component has the edge
\begin{gather}
R:=f(x_1)\asymp \phi+\phi^{1/2},\quad L:=f(x_2)\asymp \phi-\phi^{1/2}.
\end{gather}

On the other hand, for the case $M\ll N$, the critical points $x_k$'s stick to the boundary of $I_i$ in the sense that for $k=1,\dots, p$ and $C>0$

\begin{gather}
x_{2k-1}=-\sigma_k^{-1}+C\phi^{1/2},\; x_{2k}=-\sigma_k^{-1}-C\phi^{1/2},\quad f^{\prime\prime}(x_{2k-1})\asymp \phi^{-1/2},\; f^{\prime\prime}(x_{2k})\asymp -\phi^{-1/2}.
\end{gather}
Therefore, we observe that the edges of the bulk satisfy 

\begin{gather}
R_k:=f(x_{2k-1})= \sigma_k+C^{\prime}\phi^{1/2},\quad L_k:=f(x_{2k})=\sigma_k-C^{\prime}\phi^{1/2},
\end{gather}
for $C^{\prime}>0$.
\end{remark}

Following the lines around Lemma 2.6 in \cite{knowles2017anisotropic}, we have the structure of the limit of $\rho_{\mathcal{W}_o}$, say $\rho_o$.
\begin{lemma}
We have for any fixed $M,N$ satisfying \eqref{ass_M/N},
\begin{equation}
\operatorname{supp}\rho_o\cap(0,\infty)=\Big(\bigcup_{k=1}^p[L_k,R_k]\Big)\cap(0,\infty),
\end{equation}
for those non-degenerated $k$.
\end{lemma}
\begin{remark}\label{def rho}
By the relationship of the scale of $\mathcal{W}$ and $\mathcal{W}_o$, we obtain that the limiting density of $\rho_{\mathcal{W}}$, denoted as $\rho$, has the rescaled supports
\begin{equation*}
\operatorname{supp}\rho=\phi^{-1/2}\times \operatorname{supp}\rho_o.
\end{equation*}
\end{remark}
Denote $K:=\min\{M,N\}$ and $\kappa\equiv\kappa(E):=\operatorname{dist}(E,\partial\operatorname{supp}(\rho))$. We consider the following regions,
\begin{align}
\mathbf{D}\equiv\mathbf{D}(c,K):=\{z=E+\mathrm{i}\eta\in\mathbb{C}_{+}:\kappa \leq c ,K^{-1+c}\le\eta\le c^{-1}(1+\phi^{-1/2})\},
\end{align}
where $c$ is a fixed constant satisfying $0<c<1$. Based on the above discussion, we have the following theorem.
\begin{theorem}\label{thm: asym_laws}(Square root behavior)
\begin{itemize}
\item[(i)] For $z\in \mathbf{D}$ and $\phi\gtrsim 1$, we have  
\begin{gather}\label{eq_squareroot_phisim1}
m_{1}(z)\asymp1,\quad\operatorname{Im}m_{1}(z)\asymp
\begin{cases}
\sqrt{\kappa+\eta}\quad \text{if $E\in\operatorname{supp}\rho$}\\
\frac{\eta}{\sqrt{\kappa+\eta}}\quad \text{if $E\notin\operatorname{supp}\rho$}.
\end{cases}
\end{gather}
\item [(ii)] For $z\in \mathbf{D}$ and $\phi\ll1$, we have
\begin{gather}\label{eq_squareroot_phill1}
m_{1}(z)+\phi^{1/2}\sigma_k^{-1}\asymp\phi,\quad \operatorname{Im}m_{1}(z)\asymp
\phi^{3/4}.
\end{gather}
\item [(iii)]
For any $i=1,\dots,M$, we have
\begin{equation}
    |1+m_{1}(z)\sigma_i|>c,
\end{equation}
for some constant $c>0$.
\end{itemize}
\end{theorem}
\begin{proof}
The results for $\phi\asymp 1$ can be easily obtained following the lines around \cite[Lemma A.4]{knowles2017anisotropic}. For the case $\phi\gg 1$, we can observe from Remark \ref{rem: regular edges} that
\begin{equation}
\begin{split}
    z_o-R&=\sum_{l=2}^{\infty}\frac{f^{(l)}(x_1)}{l!}(m_{1o}(z_o)-x_1)^l\\
    &=\frac{f^{\prime\prime}(x_1)}{2}(m_{1o}(z_o)-x_1)^2+\rO(\phi^{1/2}|m_{1}(z)-\phi^{1/2}x_1|^3)\\
    &=\frac{\rO(\phi^{1/2})}{2}(m_{1}(z)-\phi^{1/2}x_1)^2+\rO(\phi^{1/2}|m_{1}(z)-\phi^{1/2}x_1|^3).
\end{split}
\end{equation}
By the relation of $z_o$ and $z$, the above equation implies that 
\begin{equation}
    m_{1}(z)-\phi^{1/2}x_1\asymp \sqrt{z-\widetilde{R}},
\end{equation}
where $\widetilde{R}=\phi^{-1/2}R$. Then, the results for $\phi\gg 1$ follow. For the case $\phi\ll 1$, one can derive similarly from Remark \ref{rem: regular edges} that,
\begin{equation*}
    m_1(z)-\phi^{1/2}x_k\asymp \phi^{3/4}\sqrt{\phi^{1/2} z-R_k}.
\end{equation*}
Then, by the definition of the region $\mathbf{D}$, we conclude the result for $\phi\ll1$.

For the statement $(iii)$, it is a routine for the case $\phi\gtrsim1$ using \eqref{eq_squareroot_phisim1}. One can refer to \cite{knowles2017anisotropic}. As for the case $\phi\ll1$, since $m_{1}(z)\asymp\phi^{1/2}\rightarrow0$, one may easily obtain the desired result.
\end{proof}

To avoid repetitions, we summarize the assumptions as follows.
\begin{assumption}\label{assum_summary}
We assume (A1)-(A2)-(A3), and \eqref{ass_T} hold.
\end{assumption}

Now we are ready to state the local laws.
\begin{theorem}\label{thm: locallaw_phi>1}(Local Laws)
    Suppose Assumption \ref{assum_summary} holds. When $N$ is sufficiently large, we have uniformly for $z\in\mathbf{D}$ $(\phi\gtrsim1$ or $\phi\ll 1)$,
    \begin{gather}
    \Lambda\prec\sqrt{\frac{\operatorname{Im}m_1(z)}{N\eta}}+\frac{1}{N\eta},\\
        |m_W(z)-m_1(z)|\prec\frac{1}{N\eta}.
    \end{gather}
\end{theorem}
The following result is on the rigidity of the nontrivial eigenvalues of $W$, which coincide with the nontrivial eigenvalues of $\mathcal{W}$. Let $\gamma_1>\gamma_2>\ldots>\gamma_K$ be the classical eigenvalue locations according to $\rho$ (see Remark \ref{def rho}) defined through
$$
N\int_{\gamma_i}^R\rho(\mathrm{d}x)=i.
$$
To that end, for $k=1,\ldots,p$ we define the classical number of eigenvalues in the $r$-th bulk component through
$$
N_k=N\int_{a_{2k}}^{a_{2k-1}}\rho (\mathrm{d}x). 
$$
For $k=1,\ldots,p$ and $i=1,\ldots,N_k$ we
introduce the relabellings
\begin{gather*}
    \lambda_{k,i}:={ \lambda_{i+\sum_{l<k}N_{l}}},\qquad \gamma_{k,i}:={ \gamma_{i+\sum_{l<k}N_{l}}}\in (a_{2k},a_{2k-1}).
\end{gather*}
Note that we may also characterize $\gamma_{k,i}$ through $N\int_{\gamma_{k,i}}^{a_{2k-1}} \rho(\mathrm{d}x)=i-1/2$. 

\begin{theorem}\label{thm: Rigidity} (Rigidity)
Suppose Assumption \ref{assum_summary} holds. For a sufficient small constant $c>0$, we have
\begin{equation}\label{eig rigidity}
        \bigcap_{k,i:R-c\le\gamma_{k,i}\le R}\big\{|\lambda_{k,i}-\gamma_{k,i}|\prec K^{-2/3}\cdot i^{-1/3}\big\}
\end{equation}
holds with high probability.
\end{theorem}
Beyond the support of the limiting spectrum, the statement of Theorem \ref{thm: locallaw_phi>1} may be improved to a bound that is stable all the way down to the real axis. Recall the rightmost edge of the support of $\rho$ is $R$, see Remark \ref{rem: regular edges}. For some fixed $\delta>0$, define the domain 
\begin{align}
  \mathbf{D}_{os}\equiv\mathbf{D}_{os}(c,\phi):=\{z=E+\mathrm{i}\eta\in\mathbb{C}^{+}: E-R\ge N^{-2/3+\delta}(1+\phi^{-1/2}), 0<\eta<\delta^{-1}(1+\phi^{-1/2})\}
\end{align}
of spectral parameters separated from the asymptotic spectrum by $N^{-2/3+\delta}$, which may have an arbitrarily small positive imaginary part $\eta$.
\begin{theorem}\label{thm: Local law outside the spectrum} (Local law outside the spectrum)
Suppose that Assumptions (A1)-(A2)-(A3) hold. Then
\begin{equation}\label{outside spec}
    |m_{W}(z)-m_1(z)|\prec\frac{1}{K}\frac{1}{(\kappa+\eta)+(\kappa+\eta)^2}
\end{equation}
uniformly in $z\in \mathbf{D}_{os}$. 
\end{theorem}
The proof of Theorem \ref{thm: locallaw_phi>1} is deferred to Section \ref{main proof}. The proofs of Theorem \ref{thm: Rigidity} and Theorem \ref{thm: Local law outside the spectrum} are deferred to Section \ref{Th2&3}.

\section{Proof of Theorem \ref{thm: locallaw_phi>1}}\label{main proof}
In this subsection, we give the average local law and entrywise local law for $\mathcal{W}$.

Firstly, we introduce the following control parameters,
\begin{align*}
	&\Lambda\equiv\Lambda(z):=  \max_{i,j\in{\cal I}}|G_{ij}(z)-\delta_{ij}m_{1}(z)|,\quad\widetilde{\Lambda}\equiv\widetilde{\Lambda}(z):=\max_{i,j\in{\cal I},i\ne j}|G_{ij}(z)|,\\
	&\Theta\equiv\Theta(z)  :=  |m_{W}(z)-m_{1}(z)|,
\Psi(\Theta,\phi):=\sqrt{\frac{\operatorname{Im} m_{1}(z)+\Theta}{N\eta}},\quad\Xi:=\{\Lambda\le(\log N)^{-1}\},
\end{align*}
where $\delta_{ij}$ denotes the Kronecker delta, i.e. $\delta_{ij}=1$ if $i=j$, and $\delta_{ij}=0$ if $i\ne j$ and $\Xi$ is a $z$-dependent event. For simplicity of notation, we occasionally omit the variable $z$
for those $z$-dependent quantities provided no ambiguity occurs. Also, it is convenient to define the counterpart of the above parameters for $\mathcal{W}_o$ as $\Lambda_o,\widetilde{\Lambda}_o,\Theta_o,\Psi_o(\Theta_o,\phi)$ and $\Xi_o$. 


\subsection{Case $\phi\gtrsim1$}\label{subsec: case phige1}
We begin with the case where $\phi\gtrsim1$. The first key observation is that from Theorem \ref{thm: asym_laws}, the following estimation holds uniformly on $i,j\in \{1,2,\cdots,N\}$ and $z\in\mathbf{D}$,
\begin{equation}\label{eq_priorbounds_phi>1}
    \{\mathbf{1}(\Xi)+\mathbf{1}(\eta\ge1)\}|G_{ij}^{(\mathcal{T)}}|+\mathbf{1}(\Xi)|(G^{(\mathcal{T})}_{ii})^{-1}|=\rO_{\prec}(1).
\end{equation}
Furthermore, we introduce the $Z$ variable
\begin{gather*}
    (Z_{i})^{(\mathcal{T})}:=(1-\mathbb{E}_i)\big[\mathbf{y}_i^{*}\mathcal{G}^{(i\mathcal{T})}\mathbf{y}_i\big],\quad i\notin\mathcal{T},
\end{gather*}
where $\mathbb{E}_i[\cdot]:=\mathbb{E}[\cdot|\mathcal{W}^{(i)}]$ is the partial expectation over the randomness of the $i$-th row and column of $\mathcal{W}$. It is easy to see from \eqref{ass_X} that,
\begin{equation}
    Z_{i}=\mathbf{y}_i^{*}\mathcal{G}^{(i)}\mathbf{y}_i-\frac{1}{\sqrt{MN}}\operatorname{tr}(\mathcal{G}^{(i)}\Sigma).
\end{equation}

We have the following lemma.
\begin{lemma}\label{lem: weak_locallaw_firstlemma}
    Suppose Assumption \ref{assum_summary} holds. Then uniformly for all $1\le i\le N$ and $z\in\mathbf{D}$, 
    \begin{equation*}
        \{\mathbf{1}(\Xi)+\mathbf{1}(\eta\ge1)\}(|zZ_i|+\Tilde{\Lambda})\prec\Psi(\Theta,\phi).
    \end{equation*}
\end{lemma}
\begin{proof}
The proof is similar to the one in \cite{knowles2017anisotropic} or \cite{Wen2021}; we focus on the main difference here. Let $\mathcal{I}=\{1,2,\dots, N\}$. First observe that the resolvent identities and Lemma \ref{lem: baisctool_largedevi} give that uniformly for $z\in\mathbf{D}$ and $i,j\in\mathcal{I}$ with $i\neq j$,
\begin{gather}\label{eq_entrywise_eq1}
    \mathbf{1}(\Xi)|G_{ij}|\le\mathbf{1}(\Xi)|z||G_{ii}G_{jj}^{(i)}||\mathbf{y}_i^{*}G^{(ij)}\mathbf{y}_j|\prec\mathbf{1}(\Xi)|z||G_{ii}G_{jj}^{(i)}|\frac{1}{\sqrt{MN}}\|\Sigma\|\|\mathcal{G}^{(ij)}\|_F.
\end{gather}
Using Lemma \ref{lem: basictool_resolvent}, \eqref{eq_priorbounds_phi>1} and event $\Xi$, we obtain the difference
\begin{gather*}
    \begin{split}
       \mathbf{1}(\Xi)|G_{kk}^{(ij)}-G_{kk}|\le \mathbf{1}(\Xi)C\widetilde{\Lambda}^2,
    \end{split}
\end{gather*}
which gives that 
\begin{gather}\label{eq_entrywise_diffG}
\begin{split}
\mathbf{1}(\Xi)|\operatorname{Im}{\rm tr}G^{(ij)}-\operatorname{Im}{\rm tr}G|&\le\mathbf{1}(\Xi)\Big|\sum_{k\in{\cal I}\setminus\{i,j\}}(G_{kk}^{(ij)}-G_{kk})\Big|+\mathbf{1}(\Xi)\Big|\operatorname{Im} G_{ii}+\operatorname{Im} G_{jj}\Big|\\
&\le \mathbf{1}(\Xi)CN\widetilde{\Lambda}^{2}+\mathbf{1}(\Xi)2\operatorname{Im} m_1(z)+\frac{2}{\log N}.
\end{split}
\end{gather}
Recall the following relationship
\begin{equation}
{\rm tr}{\cal G}^{(ij)}=\frac{(N-2-M)}{z}+{\rm tr}G^{(ij)}.
\end{equation}
Applying Lemma \ref{lem: basictool_resolvent}, we have
\begin{equation}
    \mathbf{1}(\Xi)\frac{\|\mathcal{G}^{(ij)}\|^2_F}{MN}=\mathbf{1}(\Xi)\frac{\operatorname{Im}{\rm tr}\mathcal{G}^{(ij)}}{\eta MN}=\mathbf{1}(\Xi)\Big(\frac{\operatorname{Im}{\rm tr}G^{(ij)}}{\eta MN}-\frac{N-2-M}{|z|^2MN}\Big).
\end{equation}
Then by \eqref{eq_entrywise_diffG} and Theorem \ref{thm: asym_laws}, we have
\begin{equation}\label{eq_entrywise_boundGF}
    \mathbf{1}(\Xi)\frac{\|\mathcal{G}^{(ij)}\|^2_F}{MN}\le \mathbf{1}(\Xi)\Big(\frac{\operatorname{Im}m_1(z)+\Theta+\widetilde{\Lambda}^2}{\eta M}+\frac{M-N}{|z|^2MN}\Big).
\end{equation}
Using \eqref{eq_priorbounds_phi>1}, we conclude from \eqref{eq_entrywise_eq1} that 
\begin{equation}
\begin{split}
    \mathbf{1}(\Xi)|G_{ij}|&\prec \mathbf{1}(\Xi)|z|\Big(\frac{\operatorname{Im}m_1(z)+\Theta+\widetilde{\Lambda}^2}{\eta M}+\frac{M-N}{|z|^2MN}\Big)^{1/2}\\
    &\prec\mathbf{1}(\Xi)\Big(\frac{\operatorname{Im}m_1(z)+\Theta+\widetilde{\Lambda}^2}{\eta N}+\frac{M-N}{MN}\Big)^{1/2}\\
    &\prec \mathbf{1}(\Xi)\Big(\frac{\operatorname{Im}m_1(z)+\Theta+\widetilde{\Lambda}^2}{\eta N}\Big)^{1/2},
\end{split}
\end{equation}
where we used the fact $|z|\asymp\phi^{1/2}$ in the second step. Therefore, by the definition of $\widetilde{\Lambda}$,
\begin{equation}\label{eq_entrywise_boundLambda}
    \mathbf{1}(\Xi)|\widetilde{\Lambda}|\prec\mathbf{1}(\Xi)\big(\frac{\operatorname{Im}m_1(z)+\Theta}{\eta N}\big)^{1/2}+\frac{\widetilde{\Lambda}}{(\eta N)^{1/2}}\quad \Rightarrow\quad\mathbf{1}(\Xi)|\widetilde{\Lambda}|\prec\mathbf{1}(\Xi)\Psi.
\end{equation}

Now, we evaluate the bound of $Z_i$. It follows from Lemmas \ref{lem: baisctool_largedevi}, \ref{lem: basictool_resolvent} and \eqref{eq_entrywise_boundGF} that 
\begin{gather}
    \begin{split}
        \mathbf{1}(\Xi)zZ_i&=\mathbf{1}(\Xi)z\Big(\mathbf{y}_i^{*}\mathcal{G}^{(i)}\mathbf{y}_i-\frac{1}{\sqrt{MN}}\operatorname{tr}(\mathcal{G}^{(i)}\Sigma)\Big)\prec\mathbf{1}(\Xi)\frac{z}{\sqrt{MN}}\|\Sigma\|\|\mathcal{G}^{(i)}\|_F\\
        &\prec z\Big(\frac{\operatorname{Im}m_1(z)+\Theta+\widetilde{\Lambda}^2}{\eta M}+\frac{M-N}{|z|^2MN}\Big)^{1/2}\prec \Big(\frac{\operatorname{Im}m_1(z)+\Theta}{\eta N}\Big)^{1/2},
    \end{split}
\end{gather}
where we used the fact $M\gtrsim N$ and \eqref{eq_entrywise_boundLambda}. The argument for $\mathbf{1}(\eta>1)$ is similar to those in $\mathbf{1}(\Xi)$, the only difference is that we will use Lemma \ref{lem: basictool_diffbounds} rather than $\widetilde{\Lambda}$ to control the difference $|\operatorname{tr}G^{(i)}-\operatorname{tr}G|$. One may refer to the details in \cite{Wen2021}.
\end{proof}

The next step is to show the following lemma,
\begin{lemma}\label{lem: weak_locallaw_seclemma}
    Suppose Assumption \ref{assum_summary} holds. Then uniformly for all $1\le i\le N$ and $z\in\mathbf{D}$, 
    \begin{equation*}
        \{\mathbf{1}(\Xi)+\mathbf{1}(\eta\ge1)\}\Big(\frac{1}{\sqrt{MN}}\operatorname{tr}(\mathcal{G}^{(i)}\Sigma)-\frac{1}{\sqrt{MN}}\operatorname{tr}\big((-z\phi^{-1/2}m_W\Sigma-zI)^{-1}\Sigma\big)\Big)\prec\Psi(\Theta,\phi).
    \end{equation*}
\end{lemma}
\begin{proof}
    We first decompose the difference into two terms,
    \begin{equation*}
       \operatorname{tr}(\mathcal{G}^{(i)}\Sigma)- \operatorname{tr}\big((-z\phi^{-1/2}m_W\Sigma-zI)^{-1}\Sigma\big)=\operatorname{tr}(\mathcal{G}^{(i)}\Sigma)-\operatorname{tr}(\mathcal{G}\Sigma)+ \operatorname{tr}(\mathcal{G}\Sigma)- \operatorname{tr}\big((-z\phi^{-1/2}m_W\Sigma-zI)^{-1}\Sigma\big).
    \end{equation*}
    Observe that for any $i\in\mathcal{I}$,
    \begin{equation}
        \begin{split}
            \mathbf{1}(\Xi)\frac{1}{\sqrt{MN}}|\operatorname{tr}(\mathcal{G}^{(i)}\Sigma)-\operatorname{tr}(\mathcal{G}\Sigma)|&=\mathbf{1}(\Xi)\frac{1}{\sqrt{MN}}|\operatorname{tr}\big((\mathcal{G}^{(i)}-\mathcal{G})\Sigma\big)|\\
            &=\mathbf{1}(\Xi)\frac{1}{\sqrt{MN}}\Big|\frac{\mathbf{y}_i^{*}\mathcal{G}^{(i)}\Sigma\mathcal{G}^{(i)}\mathbf{y}_i}{1+\mathbf{y}_{i}^{*}\mathcal{G}^{(i)}\mathbf{y}_i}\Big|\\
            &=\mathbf{1}(\Xi)\frac{1}{\sqrt{MN}}|zG_{ii}\mathbf{y}_i^{*}\mathcal{G}^{(i)}\Sigma\mathcal{G}^{(i)}\mathbf{y}_i|\\
            &\prec\mathbf{1}(\Xi)\frac{1}{\sqrt{MN}}|zG_{ii}|\big(\frac{1}{\sqrt{MN}}\operatorname{tr}(\mathcal{G}^{(i)}\Sigma\mathcal{G}^{(i)}\Sigma)+\frac{1}{\sqrt{MN}}\|\mathcal{G}^{(i)}\Sigma\mathcal{G}^{(i)}\Sigma\|_F\big)\\
            &\prec \mathbf{1}(\Xi)\frac{1}{MN}|zG_{ii}|\|\mathcal{G}^{(i)}\Sigma\|^2_F\\
            &\le \mathbf{1}(\Xi)\Big(\frac{\operatorname{Im}m_1(z)+\Theta+\widetilde{\Lambda}^2}{\eta \sqrt{MN}}\Big),
        \end{split}
    \end{equation}
where we used Lemmas \ref{lem: baisctool_largedevi}, \ref{lem: basictool_resolvent}, \eqref{eq_priorbounds_phi>1} and Theorem \ref{thm: asym_laws}.

Similarly, using the following identity
\begin{equation}
    {\cal G}-(-z\phi^{-1/2}m_{W}\Sigma-zI)^{-1}=\sum_{i\in{\cal I}}\frac{(\phi^{-1/2}m_{W}\Sigma+I)^{-1}}{z(1+\mathbf{y}_{i}^{*}{\cal G}^{(i)}\mathbf{y}_{i})}(\mathbf{y}_{i}\mathbf{y}_{i}^{*}{\cal G}^{(i)}-\frac{1}{\sqrt{MN}}\Sigma{\cal G}),
\end{equation}
one may easily obtain that
\begin{equation}\label{eq_locallaw_quadraticform1}
    \begin{split}
        &\mathbf{1}(\Xi)|\frac{1}{\sqrt{MN}}\big(\operatorname{tr}(\mathcal{G}\Sigma)- \operatorname{tr}(-z\phi^{-1/2}m_W\Sigma-zI)^{-1}\Sigma)\big)|\\
        &=\mathbf{1}(\Xi)\Big|\frac{1}{\sqrt{MN}}\operatorname{tr}\Big(\sum_{i\in\mathcal{I}}\frac{(\phi^{-1/2}m_W\Sigma+I)^{-1}}{z(1+\mathbf{y}_i^{*}\mathcal{G}^{(i)}\mathbf{y}_i)}(\mathbf{y}_i\mathbf{y}_i^{*}\mathcal{G}^{(i)}\Sigma-\frac{1}{\sqrt{MN}}\Sigma\mathcal{G}^{(i)}\Sigma+\frac{1}{\sqrt{MN}}\Sigma\mathcal{G}^{(i)}\Sigma-\frac{1}{\sqrt{MN}}\Sigma\mathcal{G}\Sigma\Big)\Big|\\
        &\prec\mathbf{1}(\Xi)\frac{1}{\sqrt{MN}}\sum_{i\in\mathcal{I}}\frac{1}{\sqrt{MN}}\|\Sigma\|\|(\phi^{-1/2}m_W\Sigma+I)^{-1}\Sigma\mathcal{G}^{(i)}\|_F+\big(\frac{\operatorname{Im}m_1(z)+\Theta+\widetilde{\Lambda}^2}{\eta \sqrt{MN}}\big)\\
        &\prec\mathbf{1}(\Xi)\frac{1}{\sqrt{MN}}\sum_{i\in\mathcal{I}}\frac{1}{\sqrt{MN}}\|(\phi^{-1/2}m_W\Sigma+I)^{-1}\|\|\mathcal{G}^{(i)}\Sigma\|_F+\big(\frac{\operatorname{Im}m_1(z)+\Theta+\widetilde{\Lambda}^2}{\eta \sqrt{MN}}\big).
    \end{split}
\end{equation}
From statement $(iii)$ of Theorem \ref{thm: asym_laws}, we have 
\begin{equation}\label{eq_locallaw_est_inverseoperatornorm}
    \mathbf{1}(\Xi)|\phi^{-1/2}m_W\sigma_i+1|\ge\mathbf{1}(\Xi)(|\phi^{-1/2}m_1\sigma_i+1|-|\phi^{-1/2}(m_1-m_W)\sigma_i|)\ge\tau^{\prime}>0.
\end{equation}
Plugging \eqref{eq_locallaw_est_inverseoperatornorm} into \eqref{eq_locallaw_quadraticform1}, we obtain that
\begin{equation}
\begin{split}
    &\mathbf{1}(\Xi)|\frac{1}{\sqrt{MN}}\big(\operatorname{tr}(\mathcal{G}\Sigma)- \operatorname{tr}(-z\phi^{-1/2}m_W\Sigma-zI)^{-1}\Sigma)\big)|\prec\frac{1}{MN}\sum_{i\in\mathcal{I}}\|\mathcal{G}^{(i)}\Sigma\|_F+\big(\frac{\operatorname{Im}m_1(z)+\Theta+\widetilde{\Lambda}^2}{\eta \sqrt{MN}}\big)\\
    &\prec \phi^{-1/2}\big(\frac{\operatorname{Im}m_1(z)+\Theta+\widetilde{\Lambda}^2}{\eta M}\big)^{1/2}+\big(\frac{\operatorname{Im}m_1(z)+\Theta+\widetilde{\Lambda}^2}{\eta \sqrt{MN}}\big),
\end{split}
\end{equation}
where we also used the estimation for $\|\mathcal{G}^{(i)}\|_F$ as in \eqref{eq_entrywise_boundGF}. Combining the above estimations, we could conclude the results for $\mathbf{1}(\Xi)$.
For the case $\mathbf{1}(\eta>1)$, the procedures are the same, so we omit further details here.
\end{proof}

With the above results, we can further prove the next lemma.
\begin{lemma}\label{lem: weak_locallaw_thdlemma}
    Suppose Assumption \ref{assum_summary} holds. Then uniformly for all $1\le i\le N$ and $z\in\mathbf{D}$, 
    \begin{equation*}
        \{\mathbf{1}(\Xi)+\mathbf{1}(\eta\ge1)\}|G_{ii}-G_{jj}|\prec\Psi(\Theta,\phi).
    \end{equation*}
\end{lemma}
\begin{proof}
    One may observe from Lemma \ref{lem: basictool_resolvent} that 
    \begin{equation}
        \begin{split}
            |G_{ii}-G_{jj}|&=\Big| G_{ii}G_{jj}\Big(\frac{1}{G_{jj}}-\frac{1}{G_{ii}}\Big)\Big|\\
            &\le|zG_{ii}G_{jj}||Z_i-Z_j|+\Big|zG_{ii}G_{jj}\Big(\frac{1}{\sqrt{MN}}\operatorname{tr}(\mathcal{G}^{(i)}\Sigma)-\frac{1}{\sqrt{MN}}\operatorname{tr}(\mathcal{G}^{(j)}\Sigma)\Big)\Big|\\
            &\prec |z||Z_i-Z_j|+|z|\big|\frac{1}{\sqrt{MN}}\operatorname{tr}(\mathcal{G}^{(i)}\Sigma-\mathcal{G}^{(j)}\Sigma)\big|\\
            &\prec\Psi+\Psi^2\prec \Psi.
        \end{split}
    \end{equation}
\end{proof}

Now, we are ready to obtain our first result, the weak local law.
\begin{lemma}[Weak local law]\label{lem: weak local law}
    Suppose Assumption \ref{assum_summary} holds. Then we have $\Lambda(z)\prec (N\eta)^{-1/4}$ uniformly for $z\in\mathbf{D}$. 
\end{lemma}
\begin{proof}
    We first observe from Lemma \ref{lem: weak_locallaw_thdlemma} that 
    \begin{equation}
        \begin{split}
           &\{\mathbf{1}(\Xi)+\mathbf{1}(\eta\ge1)\}\Big(\frac{1}{N}\sum_{i\in\mathcal{I}}\frac{1}{G_{ii}}-\frac{1}{m_W}\Big)\\
           &=\{\mathbf{1}(\Xi)+\mathbf{1}(\eta\ge1)\}\frac{1}{N}\sum_{i\in{\cal I}}\Big(-\frac{G_{ii}-m_{W}}{m_{W}^{2}}+\frac{(G_{ii}-m_{W})^{2}}{G_{ii}m_{W}^2}\Big)\\
           &=\{\mathbf{1}(\Xi)+\mathbf{1}(\eta\ge1)\}\frac{1}{N}\sum_{i\in{\cal I}}\frac{(G_{ii}-m_{W})^{2}}{G_{ii}m_{W}^2}\\
           &\prec \Psi^2.
        \end{split}
    \end{equation}
    It then follows from Lemmas \ref{lem: basictool_resolvent}, \ref{lem: weak_locallaw_firstlemma} and \ref{lem: weak_locallaw_seclemma} that 
    \begin{equation}
        \begin{split}
            &\{\mathbf{1}(\Xi)+\mathbf{1}(\eta\ge1)\}\frac{1}{m_W}=\{\mathbf{1}(\Xi)+\mathbf{1}(\eta\ge1)\}\frac{1}{N}\sum_{i\in\mathcal{I}}\frac{1}{G_{ii}}+\rO_{\prec}(\Psi^2)\\
            &=\{\mathbf{1}(\Xi)+\mathbf{1}(\eta\ge1)\}\Big(-z-\frac{z}{\sqrt{MN}}\operatorname{tr}\big((-z\phi^{-1/2}m_W\Sigma-zI)^{-1}\Sigma\big)\\
            &+\frac{z}{\sqrt{MN}}\operatorname{tr}\big((-z\phi^{-1/2}m_W\Sigma-zI)^{-1}\Sigma\big)-\frac{z}{N}\sum_{i\in\mathcal{I}}\operatorname{tr}(\mathcal{G}^{(i)}\Sigma)\Big)+\rO_{\prec}(\Psi^2)\\
            &=\{\mathbf{1}(\Xi)+\mathbf{1}(\eta\ge1)\}\Big(-z-\frac{z}{\sqrt{MN}}\operatorname{tr}\big((-z\phi^{-1/2}m_W\Sigma-zI)^{-1}\Sigma\big)\Big)+\rO_{\prec}(\Psi).
        \end{split}
    \end{equation}
    Since 
    \begin{equation*}
        \operatorname{tr}\big((-z\phi^{-1/2}m_W\Sigma-zI)^{-1}\Sigma\big)=\sum_{i\in\mathcal{I}}\frac{\sigma_i}{-z\phi^{-1/2}m_W\sigma_i-z},
    \end{equation*}
    we have 
    \begin{equation}
    \begin{split}
        &\{\mathbf{1}(\Xi)+\mathbf{1}(\eta\ge1)\}\Big(\frac{1}{m_W}+z-\frac{\phi^{1/2}}{M}\sum_{i\in\mathcal{I}}\frac{\sigma_i}{1+\phi^{-1/2}m_W\sigma_i}\Big)\prec \Psi\\
        &\Rightarrow \{\mathbf{1}(\Xi)+\mathbf{1}(\eta\ge1)\}\{f(\phi^{-1/2}m_W)-\phi^{1/2}z\}\prec \phi^{1/2}\Psi.
    \end{split}
    \end{equation}

    The following proposition gives the stability of $f(m)$.
    \begin{proposition}\label{prop: stability of f}
    Suppose Assumption \ref{assum_summary} holds. Suppose a $z$-dependent function $\delta$ satisfying $N^{-1}\le\delta(z)\le\log^{-1}N$ for $z\in\mathbf{D}$ and assume that $\delta(z)$ is Lipschitz continuous with Lipschitz constant $N^{2}$. Suppose moreover that for each fixed $E$, the function $\eta\mapsto\delta(E+\mathrm{i}\eta)$ is non-increasing for $\eta>0$. Suppose that $\mu_0:\mathbf{D}\rightarrow\mathbb{C}$ is the Stieltjes transform of a probability measure. Let $z\in\mathbf{D}$ and suppose that for all $z^{\prime}\in\operatorname{Lip}(z)$, we have $|f(\phi^{-1/2}\mu_0)-\phi^{1/2}z|\le\delta(z^{\prime})$. Then we have that for some constant $C>0$ 
    \begin{equation*}
        |\mu_0-m_{1}(z)|\le\frac{C\phi^{-1/2}\delta}{\sqrt{\kappa+\eta}+\sqrt{\phi^{-1/2}\delta}}.
    \end{equation*}
    \end{proposition}
Applying the above proposition, we have
\begin{equation}
    \{\mathbf{1}(\Xi)+\mathbf{1}(\eta\ge1)\}|m_W(z)-m_1(z)|\prec\frac{\Psi}{\sqrt{\kappa+\eta}+\sqrt{\Psi}}\prec \Psi^{1/2}.
\end{equation}
Therefore, it follows from Lemmas \ref{lem: weak_locallaw_firstlemma}, \ref{lem: weak_locallaw_seclemma} and \ref{lem: weak_locallaw_thdlemma} that 
\begin{equation}
    \mathbf{1}(\eta\ge1)\Lambda(z)\le\mathbf{1}(\eta\ge1)\Big(\max_i|G_{ii}-m_W|+|m_W-m_1|+\Tilde{\Lambda}\Big)\prec N^{-1/2}.
\end{equation}
The rest of the proof follows from a standard bootstrapping step, which can be found in \cite{Wen2021}; we omit further details here.
\end{proof}

With the above discussions, we are ready to prove Theorem \ref{thm: locallaw_phi>1} for the case $\phi\gtrsim1$.
 
\begin{proof}(Proof of Theorem \ref{thm: locallaw_phi>1} for the case $\phi\gtrsim1$)
From Lemma \ref{lem: weak local law}, we know that $\Xi$ holds with high probability, i.e. $1\prec\mathbf{1}(\Xi)$. So from now on, we can drop the factor $\mathbf{1}(\Xi)$ in all $\Xi$ dependent results without affecting their validity. Recall that 
\begin{equation}
    (1-\mathbb{E}_i)\frac{1}{G_{ii}}=-zZ_i,
\end{equation}
and we write 
\begin{equation}\label{eq_locallaw_decomp}
    \begin{split}
        &\frac{1}{\sqrt{MN}}\operatorname{tr}\big((-z\phi^{-1/2}m_W\Sigma-zI)^{-1}\Sigma\big)-\frac{1}{\sqrt{MN}}\operatorname{tr}(\mathcal{G}^{(i)}\Sigma)\\
        &=\frac{1}{\sqrt{MN}}\operatorname{tr}\Big(\sum_{i\in\mathcal{I}}\frac{(\phi^{-1/2}m_W\Sigma+I)^{-1}}{z(1+\mathbf{y}_i^{*}\mathcal{G}^{(i)}\mathbf{y}_i)}(\mathbf{y}_i\mathbf{y}_i^{*}\mathcal{G}^{(i)}\Sigma-\frac{1}{\sqrt{MN}}\Sigma\mathcal{G}^{(i)}\Sigma+\frac{1}{\sqrt{MN}}\Sigma\mathcal{G}^{(i)}\Sigma-\frac{1}{\sqrt{MN}}\Sigma\mathcal{G}\Sigma)\Big)\\
        &=\frac{1}{\sqrt{MN}}\sum_{i\in\mathcal{I}}G_{ii}\operatorname{tr}(R_i)+\frac{1}{\sqrt{MN}}\sum_{i\in\mathcal{I}}G_{ii}\frac{1}{\sqrt{MN}}\operatorname{tr}\big((\phi^{-1/2}m_W\Sigma+I)^{-1}\Sigma(\mathcal{G}^{(i)}-\mathcal{G})\Sigma\big),
    \end{split}
\end{equation}
where 
    \begin{equation}
        R_i:=(\phi^{-1/2}m_W\Sigma+I)^{-1}\big(\mathbf{y}_i\mathbf{y}_i^{*}\mathcal{G}^{(i)}\Sigma-\frac{1}{\sqrt{MN}}\Sigma\mathcal{G}^{(i)}\Sigma\big)
    \end{equation}
For the second term in \eqref{eq_locallaw_decomp}, we observe from Lemmas \ref{lem: basictool_resolvent}, \ref{lem: baisctool_largedevi}, \eqref{eq_priorbounds_phi>1} and \eqref{eq_entrywise_boundGF} that
    \begin{equation}
        \begin{split}
            &\Big|\frac{1}{\sqrt{MN}}\sum_{i\in\mathcal{I}}G_{ii}\frac{1}{\sqrt{MN}}{\rm tr}((\phi^{-1/2}m_W\Sigma+I)^{-1}\Sigma(\mathcal{G}^{(i)}-\mathcal{G})\Sigma)\Big|\\
	&=\Big|\frac{1}{\sqrt{MN}}\sum_{i\in\mathcal{I}}G_{ii}\frac{1}{\sqrt{MN}}{\rm tr}((\phi^{-1/2}m_W\Sigma+I)^{-1}\Sigma\frac{\mathcal{G}^{(i)}\mathbf{y}_i\mathbf{y}_i^{*}\mathcal{G}^{(i)}}{1+\mathbf{y}_i^{*}\mathcal{G}^{(i)}\mathbf{y}_i}\Sigma)\Big| \\
	&\le\frac{1}{\sqrt{MN}}\sum_{i\in\mathcal{I}}|z||G_{ii}|^2\frac{1}{\sqrt{MN}}|\mathbf{y}_i^{*}\mathcal{G}^{(i)}\Sigma(\phi^{-1/2}m_W\Sigma+I)^{-1}\Sigma\mathcal{G}^{(i)}\mathbf{y}_i|\\
 &\prec \phi^{-1/2}\big(\frac{\operatorname{Im}m_1(z)+\Theta+\Tilde{\Lambda}^2}{\eta\sqrt{MN}}\big)\prec \Psi^2.
        \end{split}
    \end{equation}
Furthermore, by the same procedures, one can easily check that 
    \begin{equation}
        \begin{split}
            \frac{1}{\sqrt{MN}}\sum_{i\in\mathcal{I}}G_{ii}\operatorname{tr}(R_i)&=\frac{1}{\sqrt{MN}}\sum_{i\in\mathcal{I}}G_{ii}\operatorname{tr}(R^{(i)}_i)+\frac{1}{\sqrt{MN}}\sum_{i\in\mathcal{I}}G_{ii}\operatorname{tr}(R_i-R^{(i)}_i)\\
            &\prec\frac{1}{\sqrt{MN}}\sum_{i\in\mathcal{I}}G_{ii}\operatorname{tr}(R^{(i)}_i)+\phi^{-1/2}(\frac{\operatorname{Im}m_1(z)+\Theta+\Tilde{\Lambda}^2}{\eta M})^{1/2}\frac{\phi^{-1/2}}{N\eta},
        \end{split}
    \end{equation}
where we used Lemma \ref{lem: basictool_diffbounds} and $R_i^{(i)}:=(\phi^{-1/2}m_W^{(i)}\Sigma+I)^{-1}\big(\mathbf{y}_i\mathbf{y}_i^{*}\mathcal{G}^{(i)}\Sigma-\frac{1}{\sqrt{MN}}\Sigma\mathcal{G}^{(i)}\Sigma\big)$.

To improve the weak local law to the strong local law, a key input is Proposition \ref{prop: fluctuation averaging} below, whose proof can be found in \cite{Wen2021}.
 \begin{proposition}\label{prop: fluctuation averaging}
Suppose Assumption \ref{assum_summary} holds.  Let $\nu\in[1/4,1]$. Denote $\Phi_{\nu}=\sqrt{\frac{\operatorname{Im} m_1(z)+(N\eta)^{-\nu}}{N\eta}}+\frac{1}{N\eta}.$
Suppose moreover that $\Theta\prec(N\eta)^{-\nu}$ uniformly for $z\in\mathbf{D}$.
Then we have
	\begin{equation}
		\frac{1}{\sqrt{MN}}\sum_{i\in{\cal I}}(1-\mathbb{E}_i)\frac{1}{G_{ii}}\prec\Phi_{\nu}^{2},\label{eq:fluctuation of Z}
		\end{equation}
		and
		\begin{equation}
		\frac{1}{\sqrt{MN}}\sum_{i\in{\cal I}}(1-\mathbb{E}_i)\mathscr{R}_{i}\prec \Phi_{\nu}^{2},\label{eq:fluctuation of R}
		\end{equation}
		uniformly for $z\in\mathbf{D}$, where
		\begin{equation}
		\mathscr{R}_i:=\mathbf{y}_i^{*}\mathcal{G}^{(i)}\Sigma(\phi^{-1/2}m_W^{(i)}\Sigma+I)^{-1}\mathbf{y}_i.
		\end{equation}
 \end{proposition}

Therefore, given that $\Theta\prec(N\eta)^{-\nu}$ for some $\nu\in[1/4,1]$, it follows from Proposition \ref{prop: fluctuation averaging} that
\begin{equation}\label{rigidi}
    |f(\phi^{-1/2}m_W)-\phi^{1/2}z|\prec \phi^{1/2}\Psi_{\nu}^2\prec \phi^{1/2}\big(\frac{1}{(N\eta)^{\nu+1}}+\frac{\operatorname{Im}m_1(z)}{N\eta}\big).
\end{equation}
Then we observe from Proposition \ref{prop: stability of f} and Theorem \ref{thm: asym_laws} that 
\begin{equation}
    \Theta\prec\frac{1}{(N\eta)^{(\nu+1)/2}}.
\end{equation}
So, 
\begin{equation}
    \Lambda\prec\Psi_{\nu}^2+\Theta\prec\Big(\sqrt{\frac{\operatorname{Im}m_1(z)}{N\eta}}+\frac{1}{(N\eta)^{(\nu+1)/2}}\Big).
\end{equation}
One can see that the error bound of $\Lambda$ improves from $1/(N\eta)^{\nu}$ to $1/(N\eta)^{(\nu+1)/2}$. Hence implementing the above argument a finite number of times, we obtain 
\begin{equation}
    \Lambda\prec \sqrt{\frac{\operatorname{Im}m_1(z)}{N\eta}}+\frac{1}{N\eta}.
\end{equation}
We complete the proof for the case $\phi\gtrsim1$. 
\end{proof}

\subsection{Case $\phi\ll 1$}
In this subsection, we consider the local laws for the case $\phi\ll 1$. We inherit the notation from Subsection \ref{subsec: case phige1}. Since, at this time, the scale of the typical rates of $\mathcal{W}$ is pretty large, saying $\rO(\phi^{-1/2})$, we turn to the discussion with $\mathcal{W}_o$ and scale back in the end. We first denote one $o-$region that
\begin{align*}
\mathbf{D}_o:=\phi^{1/2}\mathbf{D}.
\end{align*}
Denote $\eta_o$ as the imaginary part in $\mathbf{D}_o$ to avoid confusion. We have
\begin{align}\label{etao}
    \eta_o=\phi^{1/2}\eta.
\end{align}
Then, similar to \eqref{eq_priorbounds_phi>1}, we have the following observation,
\begin{equation}
    \{\mathbf{1}(\Xi_o)+\mathbf{1}(\eta_o\ge1)\}|(G_o)_{ij}^{(\mathcal{T})}|+\mathbf{1}(\Xi_o)|((G_o)_{ii}^{(\mathcal{T})})^{-1}|=\rO_{\prec}(1),
\end{equation}
which holds uniformly on $i,j\in\mathcal{I}$ and $z_o\in\mathbf{D}_o$. Denote 
\begin{equation}
    (Z_o)_i:=(1-\mathbb{E}_i)[(\mathbf{y}_i^o)^{*}\mathcal{G}_o^{(i)}\mathbf{y}_i^o]=(\mathbf{y}_i^o)^{*}\mathcal{G}_o^{(i)}\mathbf{y}_i^o-\frac{1}{N}\operatorname{tr}(\mathcal{G}_o^{(i)}\Sigma).
\end{equation}
We have the following lemma parallel to Lemma \ref{lem: weak_locallaw_firstlemma},
\begin{lemma}\label{lem: weak_firstlemma_phi<1}
     Suppose Assumption \ref{assum_summary} holds. Then uniformly for all $1\le i\le N$ and $z_o\in\mathbf{D}_o$, 
    \begin{equation*}
        \{\mathbf{1}(\Xi_{o})+\mathbf{1}(\eta_o\ge1)\}(|z_o(Z_o)_i|+\Tilde{\Lambda}_o)\prec\Psi_o(\Theta_o,\phi).
    \end{equation*}
\end{lemma}
\begin{proof}
    The proof of this lemma is similar to the one in Lemma \ref{lem: weak_locallaw_firstlemma} or \cite{Wen2021}; we only indicate several key observations. Firstly, We still have
    \begin{equation}
        \mathbf{1}(\Xi_o)|\operatorname{Im}\operatorname{tr}(G_o)^{(ij)}-\operatorname{Im}\operatorname{tr}G_o|\le\mathbf{1}(\Xi_o)\{CN\Tilde{\Lambda}_o^2+2\operatorname{Im}m_{1o}(z_o)\}.
    \end{equation}
    It then, together with the fact
    \begin{equation}
        \mathbf{1}(\Xi_o)\frac{\|(\mathcal{G}_o)^{(ij)}\|^2_F}{N^2}=\mathbf{1}(\Xi_o)\Big(\frac{\operatorname{Im}\operatorname{tr}G_o^{(ij)}}{\eta_o N^2}-\frac{N-2-M}{|z_o|^2N^2}\Big),
    \end{equation}
    implies that 
    \begin{equation}
        \mathbf{1}(\Xi_o)\frac{\|(\mathcal{G}_o)^{(ij)}\|^2_F}{N^2}\le\mathbf{1}(\Xi_o)\Big(\frac{\operatorname{Im}m_{1o}(z_o)+\Theta_o+\Tilde{\Lambda}^2_o}{\eta_o N}+\frac{1}{N}\Big).
    \end{equation}
    Then, it is a routine to obtain the desired results.
\end{proof}

By the procedures in the proof of Lemma \ref{lem: weak_firstlemma_phi<1}, we obtain the following result, which is the counterpart of Lemmas \ref{lem: weak_locallaw_seclemma} and \ref{lem: weak_locallaw_thdlemma},
\begin{lemma}\label{lem: weak_seclemma_phi<1}
    Suppose Assumption \ref{assum_summary} holds. Then uniformly for all $1\le i\le N$ and $z_o\in\mathbf{D}_o$, 
    \begin{equation*}
        \{\mathbf{1}(\Xi_o)+\mathbf{1}(\eta_o\ge1)\}\Big(\frac{1}{N}\operatorname{tr}(\mathcal{G}_o^{(i)}\Sigma)-\frac{1}{N}\operatorname{tr}\big((-z_om_{W_o}\Sigma-z_oI)^{-1}\Sigma\big)\Big)\prec\Psi_o(\Theta_o,\phi).
    \end{equation*}
\end{lemma}
\begin{lemma}\label{lem: weak_thdlemma_phi<1}
    Suppose Assumption \ref{assum_summary} holds. Then uniformly for all $1\le i\le N$ and $z_o\in\mathbf{D}_o$, 
    \begin{equation*}
        \{\mathbf{1}(\Xi_o)+\mathbf{1}(\eta_o\ge1)\}|(G_o)_{ii}-(G_o)_{jj}|\prec\Psi_o(\Theta_o,\phi).
    \end{equation*}
\end{lemma}
\begin{proof}
    The proof of the above two lemmas is a standard one, which directly follows from the bound for $\|\mathcal{G}_o^{(i)}\Sigma\|_F$ in the proof of Lemma \ref{lem: weak_firstlemma_phi<1}. One can refer to \cite{Wen2021} for more details.
\end{proof}

Combining the above results, we have the following weak local law.
\begin{lemma}[Weak local law]\label{lem: weak_phi<1}
    Suppose Assumption \ref{assum_summary} holds. Then we have $\Lambda_o(z_o)\prec (N\eta_o)^{-1/4}$ uniformly for $z_o\in\mathbf{D}_o$. 
\end{lemma}
\begin{proof}
    Lemmas \ref{lem: weak_firstlemma_phi<1}-\ref{lem: weak_thdlemma_phi<1} imply that 
    \begin{equation}
        \{\mathbf{1}(\Xi_o)+\mathbf{1}(\eta_o\ge 1)\}\Big(\frac{1}{m_{W_o}}+z_o-\frac{\phi}{N}\sum_{i\in\mathcal{I}}\frac{\sigma_i}{1+m_{W_o}\sigma_i}\Big)\prec\Psi_o.
    \end{equation}
    Then by the stability of $f(x)$, one has
    \begin{equation}
        \{\mathbf{1}(\Xi_o)+\mathbf{1}(\eta_o\ge 1)\}|m_{W_o}(z_o)-m_{1o}(z_o)|\prec \Psi_o^{1/2}.
    \end{equation}
    Then, after deploying a standard bootstrapping argument, this lemma follows.
\end{proof}

Now we state the local laws for $\mathcal{W}_o$,
\begin{theorem}
    Suppose Assumption \ref{assum_summary} holds. When $N$ is sufficiently large, we have uniformly for $z_o\in\mathbf{D}_o$,
    \begin{gather}
        \Lambda_o\prec\sqrt{\frac{\operatorname{Im}m_{1o}(z_o)}{N\eta_o}}+\frac{1}{N\eta_o},\\
        |m_{W_o}(z_o)-m_{1o}(z_o)|\prec\frac{1}{N\eta_o}.
    \end{gather}
\end{theorem}
\begin{proof}
    The proof of this theorem consists of a standard argument as the one in case $\phi \gtrsim 1$ based on the fluctuation averaging properties as in Proposition \ref{prop: fluctuation averaging}. One may refer to \cite{Wen2021} for more details.
\end{proof}

Finally, according to the relationship between $\eta_o$ and $\eta$, i.e. \eqref{etao}, and the relationship between $m_{1o}(z_o)$ and $m_1(z)$, i.e. \eqref{m1ozo}, the eigenvalues of $W_o$ and $W$, we obtain uniformly for $z\in\mathbf{D}$,
    \begin{gather*}
    \Lambda\prec\phi^{1/2}\Big(\sqrt{\frac{\operatorname{Im}m_{1o}(z_o)}{N\eta_o}}+\frac{1}{N\eta_o}\Big)=\sqrt{\frac{\phi^{1/2}\operatorname{Im}m_{1}(z)}{N\eta_o}}+\frac{\phi^{1/2}}{N\eta_o}=\sqrt{\frac{\operatorname{Im}m_{1}(z)}{N\eta}}+\frac{1}{N\eta},\\
        |m_{W}(z)-m_{1}(z)|\prec\frac{\phi^{1/2}}{N\eta_o}=\frac{1}{N\eta}.
    \end{gather*}
Together with the case $\phi\gtrsim 1$, we finish the proof of Theorem \ref{thm: locallaw_phi>1}.

\section{Proofs of Theorem \ref{thm: Rigidity} and Theorem \ref{thm: Local law outside the spectrum}}\label{Th2&3}
\subsection{Proof of Theorem \ref{thm: Rigidity}}
We first show that there is no eigenvalue outside the spectrum with high probability, i.e.
\begin{equation}\label{eq:rigidity of largest eigenvalue}
    \lambda_{1,1}=\lambda_1\le R+\mathrm{O}_{\prec}(N^{-2/3}).
\end{equation}

Rewrite \eqref{rigidi} without scale $\phi^{-1/2}$, we obtain
\begin{equation}
    |f(m_W)-z| \prec \big(\frac{1}{(N\eta)^{2}}+\frac{\operatorname{Im}m_1(z)}{N\eta}\big).
\end{equation}
uniformly for $z\in\mathbf{D}$ when $\phi\gtrsim 1$ or $\phi \ll 1$, respectively. In the sequel, we only give the details for the case $\phi\gtrsim 1$ , while the case $\phi\ll1$ follows from similar argument.

Now, we obtain from Proposition \ref{prop: stability of f} that, for any
$\varepsilon,D>0$, as $N$ is sufficiently large,
\begin{multline*}
	\sup_{z\in\mathbf{D}}\mathbb{P}\Big(|m_{W}-m_1|>\frac{N^{\varepsilon}}{\sqrt{\kappa+\eta}}(\frac{\operatorname{Im} m_1}{N\eta}+\frac{1}{(N\eta)^{2}})\Big)\\
	\le\sup_{z\in\mathbf{D}}\mathbb{P}\Big(|m_{W}-m_1|>\frac{N^{\varepsilon}(\frac{\operatorname{Im} m_1}{N\eta}+\frac{1}{(N\eta)^{2}})}{\sqrt{\kappa+\eta}+\sqrt{N^{\varepsilon/2}(\frac{\operatorname{Im} m_1}{N\eta}+\frac{1}{(N\eta)^{2}})}}\Big)\le N^{-D},
\end{multline*}
so uniformly for $z\in\mathbf{D}$,
\begin{equation}
	|m_{W}-m_1|\prec\frac{1}{\sqrt{\kappa+\eta}}(\frac{\operatorname{Im} m_1}{N\eta}+\frac{1}{(N\eta)^{2}}).\label{eq:improved bound of m-m}
\end{equation}

Since we already know that $\lambda_{1}\le \phi+C_1\phi^{1/2}$ with high probability for some $C_1>0$. Therefore, in order to show \eqref{eq:rigidity of largest eigenvalue},
it remains to show that for any fixed $\varepsilon>0$, there is no
eigenvalue of $W$ in the interval
\begin{equation}
	\mathbf{I}:=[R+N^{-2/3+4\varepsilon},\phi+C_1\phi^{1/2}]\label{eq:definition of I}
\end{equation}
with high probability. The idea of the proof is to choose, for each
$E\in\mathbf{I}$, a scale $\eta(E)$ such that $\operatorname{Im} m_{W}(E+\mathrm{i}\eta(E))\le\frac{N^{-\varepsilon}}{N\eta(E)}$
with high probability.

Let $\varepsilon$ be as in (\ref{eq:definition of I}). Then we observe
from (\ref{eq:improved bound of m-m}) and Lemma \ref{thm: asym_laws} that
\begin{equation}
	\bigcap_{z\in\mathbf{D},E\ge R}\Big\{|m_{W}(z)-m_1(z)|\le N^{\varepsilon}\Big(\frac{\eta}{\kappa}\frac{1}{N\eta}+\frac{1}{\sqrt{\kappa}}(\frac{1}{(N\eta)^{2}})\Big)\Big\}\label{eq:proof of rigidity bound 1}
\end{equation}
holds with high probability.

For each $E\in\mathbf{I}$, we define
\begin{equation}
	\eta(E)=N^{-1/2-\varepsilon}\kappa(E)^{1/4},\qquad z(E)=E+\mathrm{i}\eta(E).\label{eq:choice of eta(E)}
\end{equation}
One may see that we still have $z\in\mathbf{D}$.
			
Using Lemma \ref{thm: asym_laws}, we find that for all $E\in\mathbf{I}$
\begin{equation}
	\operatorname{Im} m_1(z(E))\le\frac{\eta(E)}{\sqrt{\kappa(E)}}\le\frac{N^{-\varepsilon}}{N\eta(E)}.\label{eq:rigidity Im m bound}
\end{equation}

With the choice $\eta(E)$ in (\ref{eq:choice of eta(E)}), we obtain
from (\ref{eq:proof of rigidity bound 1}) that
\begin{equation}
	\bigcap_{E\in\mathbf{I}}\Big\{|m_{W}(z)-m_1(z)|\le\frac{2N^{-\varepsilon}}{N\eta(E)}\Big\}\text{ holds with high probability}.\label{eq:rigidity mN-m bound}
\end{equation}
			
From (\ref{eq:rigidity Im m bound}) and (\ref{eq:rigidity mN-m bound})
we conclude that
\begin{equation}
	\bigcap_{E\in\mathbf{I}}\Big\{\operatorname{Im} m_{W}(z)\le\frac{3N^{-\varepsilon}}{N\eta(E)}\Big\}\text{ holds with high probability}.\label{eq:proof of rigidity intersection bounds of Imm}
\end{equation}
Now suppose that there is an eigenvalue, say $\lambda_{i}$ of $W$
in $\mathbf{I}$. Then we find that
	\[
	    \operatorname{Im} m_{W}(z(\lambda_{i}))=\frac{1}{N}\sum_{j}\frac{\eta(\lambda_{i})}{(\lambda_{j}-\lambda_{i})^{2}+\eta(\lambda_{i})^{2}}\ge\frac{1}{N\eta(\lambda_{i})},
	\]
which contradicts with the inequality in (\ref{eq:proof of rigidity intersection bounds of Imm}).
Therefore, we conclude that with high probability, there is no eigenvalue
in $\mathbf{I}$. Since $\varepsilon>0$ in (\ref{eq:definition of I})
is arbitrary, (\ref{eq:rigidity of largest eigenvalue}) follows. We remark that in the case $\phi\ll1$, we consider the matrix $W_o$ instead, which gives the right scaling for $z_o$. Then, repeating the above argument, we obtain the parallel result for \eqref{eq:proof of rigidity intersection bounds of Imm} that $\operatorname{Im}m_{W_o}(z_o)\prec \frac{N^{1/2+\epsilon}}{N}$. On the other hand, we have $\operatorname{Im}m_{W_o}(z_o)\ge \frac{N^{1/2+1/8+\epsilon}}{N}$. Therefore, there is still a contradiction.

Now we turn to show \eqref{eig rigidity}. Define $\mathfrak{n}_{N}(a,b)=\int_{a}^{b}\rho_{W}({\rm d}x)$
		and $\mathfrak{n}(a,b)=\int_{a}^{b}\rho({\rm d}x)$ for any $a\le b\in\mathbb{R}$. It is easy to see that $\mathfrak{n}_{N}(a,b)$ and $\mathfrak{n}(a,b)$ count the number of sample eigenvalues and population eigenvalues in $[a,b]$, respectively. We need the following result
\begin{lemma}
			\label{lem:integral of rhoDelta bound}Let $a_{1},a_{2}$ be two numbers
			with $a_{1}\le a_{2}$ and $|a_{1}|+|a_{2}|=\mathrm{O}(1)$. For any $E_{1},E_{2}\in[a_{1},a_{2}]$
			and $\eta=N^{-1}$, let $\psi(\lambda):=\psi_{E_{1},E_{2},\eta}(\lambda)$
			be a $C^{2}(\mathbb{R})$ function such that $\psi(x)=1$ for $x\in[E_{1}+\eta,E_{2}-\eta]$,
			$\psi(x)=0$ for $x\in\mathbb{R}\backslash[E_{1},E_{2}]$ and the
			first two derivatives of $\psi$ satisfy $|\psi^{(1)}(x)|\le C\eta^{-1}$,
			$|\psi^{(2)}(x)|\le C\eta^{-2}$ for all $x\in\mathbb{R}$. Let $\varrho^{\Delta}$
			be a signed measure on the real line and $m^{\Delta}$ be the Stieltjes
			transform of $\varrho^{\Delta}$. Suppose, for some positive number
			$c_{N}$ depending on $N$, we have
			\begin{equation}
			|m^{\Delta}(x+\mathrm{i} y)|\le Cc_{N}(\frac{1}{Ny}+\frac{1}{\sqrt{\kappa+y}\cdot(Ny)^2}) \qquad\forall y<1,x\in[a_{1},a_{2}].\label{eq:mDelta bound}
			\end{equation}
			Then
			\begin{equation}
			\begin{split}
			\Big|\int\psi(\lambda)\varrho^{\Delta}(\mathrm{d}\lambda)\Big|&\leqslant c_N(\frac{1}{N}+\frac{4}{\sqrt{\kappa+1/2}\cdot N^2})\\
			&\leqslant c_N(\frac{1}{N}+\frac{\sqrt{E_2-E_1+\eta}}{N^2})\\
			&\leqslant c_N(\frac{1}{N}+\frac{\sqrt{\kappa_{E_1}}}{N^2}).
			\end{split}
			\end{equation}
		\end{lemma}
  \begin{proof}
      The proof of Lemma \ref{lem:integral of rhoDelta bound} is a standard one by applying the Helffer-Sj\"{o}strand formula on the signed measure $\varrho^{\Delta}$ with its Stieltjes transform of $\varrho^{\Delta}$. One may refer to the proof of \cite[Theorem 3.2]{Wen2021} for more details. We omit them for simplification.
  \end{proof}
  Adopt Lemma \ref{lem:integral of rhoDelta bound} with $\varrho^{\Delta}$ being the signed measure $\rho_W-\rho$ and notice that the scaling factor of $\rho_W$ is $\sqrt{MN}$. Since the typical order of the eigenvalues of $W$ is $\mathrm{O}(\phi^{1/2})$, one may gradually change the interval in any integrals involving $\rho_W$ as
  \begin{gather*}
      \int_a^b\rho_W(\mathrm{d}\lambda)=\int_{\phi^{-1/2}a}^{\phi^{-1/2}b}\phi^{1/2}\rho_W(\mathrm{d}(\phi^{-1/2}\lambda))
  \end{gather*}
  to rescale the interval $[a,b]$ to constant order. One may observe that this makes the scaling factor of $\phi^{1/2}\rho_W$ becoming $N^{-1}$. A similar pattern will be followed by $\rho$. In the sequel, we will repeatedly use this rescaled integral without further illustration. Suppose $y\ge y_{0}=N^{-1+\tau}$ for some small constant $\tau>0$ (Since at this time $K=N$, we may choose such $y_0$ in $\mathbf{D}$). By Theorem \ref{thm: locallaw_phi>1}, one has the condition in Lemma \ref{lem:integral of rhoDelta bound} holds with high probability for the difference $m^{\Delta}=m_W-m_1$ and $c_{N}=N^{\varepsilon}$ for any small $\varepsilon>0$. For $y\le y_{0}$, set $z=x+\mathrm{i} y$, $z_{0}=x+\mathrm{i} y_{0}$ and
		estimate
		\begin{equation}
		|m_{W}(z)-m_1(z)|\le|m_{W}(z_{0})-m_1(z_{0})|+\int_{y}^{y_{0}}\Big|\frac{\partial}{\partial\eta}\{m_{W}(x+\mathrm{i}\eta)-m_1(x+\mathrm{i}\eta)\}\Big|\mathrm{d}\eta.\label{eq:bound of m_N-m by int_y^y0}
		\end{equation}

		Note that
		\begin{multline*}
		\Big|\frac{\partial}{\partial\eta}m_{W}(x+\mathrm{i}\eta)\Big|=\Big|\frac{\partial}{\partial\eta}\int\frac{1}{\lambda-x-\mathrm{i}\eta}\rho_W({\rm d}\lambda)\Big|\\
		\le\int\frac{1}{|\lambda-x-\mathrm{i}\eta|^{2}}\rho_W({\rm d}\lambda)=\eta^{-1}\operatorname{Im} m_{W}(x+\mathrm{i}\eta).
		\end{multline*}
		The same bound applies to $|\frac{\partial}{\partial\eta}m_1(x+\mathrm{i}\eta)|$
		with $m_{W}$ replaced by $m_1$.

  Using Theorem \ref{thm: locallaw_phi>1} and the fact that the
		functions $y\to y\operatorname{Im} m_{W}(x+\mathrm{i} y)$ and $y\to y\operatorname{Im} m_1(x+\mathrm{i} y)$
		are both increasing for $y>0$ since both are Stieltjes
		transforms of a positive measure, we obtain that
		\begin{equation*}
		    \begin{split}
		   &\int_{y}^{y_{0}}\Big|\frac{\partial}{\partial\eta}\{m_{W}(x+\mathrm{i}\eta)-m_1(x+\mathrm{i}\eta)\}\Big|\mathrm{d}\eta \\
          &\int_{y}^{y_{0}}\frac{1}{\eta}\{\operatorname{Im} m_{W}(x+\mathrm{i}\eta)+\operatorname{Im} m_1(x+\mathrm{i}\eta)\}\mathrm{d}\eta\\
			& \le y_{0}\{\operatorname{Im} m_{W}(z_{0})+\operatorname{Im} m_1(z_{0})\}\int_{y}^{y_{0}}\frac{1}{\eta^{2}}\mathrm{d}\eta\\
			& = y_{0}\{\operatorname{Im} m_{W}(z_{0})+\operatorname{Im} m_1(z_{0})\}(\frac{1}{y}-\frac{1}{y_{0}})\\
			& = \{\operatorname{Im} m_{W}(z_{0})+\operatorname{Im} m_1(z_{0})\}\frac{y_{0}-y}{y}\\
			& \prec 2\operatorname{Im} m_1(z_{0})+(Ny_{0})^{-1}.
		    \end{split}
		\end{equation*}
	
		Hence we have from (\ref{eq:bound of m_N-m by int_y^y0}) that
		\begin{equation}
		|m_{W}(z)-m_1(z)|\prec2\operatorname{Im} m_1(z_{0})+(Ny_{0})^{-1}\le\frac{CNy+1}{Ny}\le\frac{CN^{\tau}}{Ny}.\label{eq:check the bound for lemma below}
		\end{equation}

    Let $\psi_{E_{1},E_{2},\eta}$ be the function in Lemma \ref{lem:integral of rhoDelta bound}.
		Applying Lemma \ref{lem:integral of rhoDelta bound} with $c_{N}=N^{\tau}$,
		we obtain that for any $\eta=N^{-1}$
		\[
		\Big|\int_{\mathbb{R}}\psi_{E_{1},E_{2},\eta}(\lambda)\rho_{W}(\mathrm{d}\lambda)-\int_{\mathbb{R}}\psi_{E_{1},E_{2},\eta}(\lambda)\rho(\mathrm{d}\lambda)\Big|\prec N^{-1+\tau}.
		\]
		
		Integrating with respect to $\rho_{W}({\rm d}\lambda)$ and $\rho({\rm d}\lambda)$
		both sides of the following elementary inequality
		\[
		\mathbf{1}_{[x-\eta,x+\eta]}(\lambda)\le\frac{2\eta^{2}}{(\lambda-x)^{2}+\eta^{2}}\quad\forall x,\lambda\in\mathbb{R},
		\]
		and using (\ref{eq:check the bound for lemma below}), Theorem \ref{thm: asym_laws}
		and the definitions of $y_{0}$, we get that for some constant $C>0$
		\[
		\mathfrak{n}_{N}(x-\eta,x+\eta)\le C\eta\operatorname{Im} m_{W}(x+\mathrm{i}\eta)\le Cy_{0}\operatorname{Im} m_W(x+\mathrm{i} y_{0})\prec N^{-1+\tau},
		\]
		and
		\[
		\mathfrak{n}(x-\eta,x+\eta)\le C\eta\operatorname{Im} m_1(x+\mathrm{i}\eta)\le Cy_{0}\operatorname{Im} m_1(x+\mathrm{i} y_{0})\prec N^{-1+\tau},
		\]
		uniformly for $x$ in a small neighborhood of $R$. Then, we obtain that
		\begin{eqnarray*}
    			|\mathfrak{n}_{N}(E_{1},E_{2})-\mathfrak{n}(E_{1},E_{2})| & = & |\int_{E_{1}}^{E_{2}}\rho_{W}({\rm d}\lambda)-\int_{E_{1}}^{E_{2}}\rho({\rm d}\lambda)|\\
			& \le & |\int_{E_{1}+\eta}^{E_{2}-\eta}\psi_{E_{1},E_{2},\eta}(\lambda)\rho_{W}({\rm d}\lambda)-\int_{E_{1}+\eta}^{E_{2}-\eta}\psi_{E_{1},E_{2},\eta}(\lambda)\rho({\rm d}\lambda)|\\
			&  & +|\int_{E_{1}}^{E_{1}+\eta}\rho_{W}({\rm d}\lambda)|+|\int_{E_{1}}^{E_{1}+\eta}\rho({\rm d}\lambda)|\\
			&  & +|\int_{E_{2}-\eta}^{E_{2}}\rho_{W}({\rm d}\lambda)|+|\int_{E_{2}-\eta}^{E_{2}}\rho({\rm d}\lambda)|\\
			& \prec & \frac{1}{N^{1-\tau}}+\frac{1}{N^{1-\tau}}(\sqrt{\kappa_{E_1}}-\sqrt{\kappa_{E_2}}).
		\end{eqnarray*}
    Notice that $\tau>0$ is arbitrary, we have that if $\lambda_{r,i},\gamma_{r,i}\geqslant R-N^{c}N^{-2/3}$ for some $c>0$, then, with high probability
			\begin{equation}\label{eq:x near the edge}
			|\lambda_{r,i}-\gamma_{r,i}|\leqslant N^{-\epsilon}N^{-2/3},
			\end{equation}
			for some $\epsilon>0$. By the square root behavior of $\rho$, we have $\mathfrak{n}(x)\sim(\lambda_{1,1}-x)^{3/2}$ when $x$ is near the edge. That is
			\[
			\mathfrak{n}(\gamma_{r,i})=\frac{i}{N}\sim(\lambda_{1,1}-\gamma_{r,i})^{3/2}.
			\]
			
			Thus, we have proved the case for $\lambda_{r,i}$ near the edge. Together with (\ref{eq:x near the edge}), we conclude (\ref{eig rigidity}). For the case where $R-\lambda_{r,i}>N^cN^{-2/3}$ and $R-\gamma_{r,i}>N^cN^{-2/3}$, by the definition of $\mathfrak{n}$, one may check that for sufficient large $E_3>R$,
   \begin{equation*}
       \mathfrak{n}(\gamma_{r,i},E_3)\vee\mathfrak{n}(\lambda_{r,i},E_3)\gtrsim (N^cN^{-2/3})^{3/2}\ge \frac{N^{c_1}}{N},
   \end{equation*}
   for some $c_1=3c/2>\tau$. Then by the definition of $\rho_W$ and $\gamma_{r,i}$, one has
   \begin{equation}\label{eq_eigenvaluerigidity_approximation_n}
       \frac{i}{N}=\mathfrak{n}(\gamma_{r,i},E_3)+\mathrm{O}(\frac{1}{N})=\mathfrak{n}_N(\lambda_{r,i},E_3)=\mathfrak{n}(\lambda_{r,i},E_3)+\mathrm{O}(\frac{N^{\tau}}{N}).
   \end{equation}
   So, 
   \begin{equation}
       \mathfrak{n}(\gamma_{r,i},E_3)=\mathfrak{n}(\lambda_{r,i},E_3)(1+\mathrm{O}(N^{-c/2}))
   \end{equation}
   with high probability. Again, by the square root behavior of $\rho$, one has $\mathfrak{n}(x,E_3)\asymp(R-z)^{3/2}$. Then we deduce that $R-\lambda_{r,i}\asymp R-\gamma_{r,i}$ with high probability. Moreover, one may check that $\partial\mathfrak{n}(\lambda_{r,i},E_3)/\partial \lambda_{r,i}\asymp \partial\mathfrak{n}(\gamma_{r,i},E_3)/\partial \gamma_{r,i}$ with $\partial\mathfrak{n}(x,E_3)/\partial x\asymp (R-x)^{1/2}$. Then it follows from the mean value theorem and \eqref{eq_eigenvaluerigidity_approximation_n} that 
   \begin{equation*}
       |\lambda_{r,i}-\gamma_{r,i}|\asymp\frac{|\mathfrak{n}(\lambda_{r,i},E_3)-\mathfrak{n}(\gamma_{r,i},E_3)|}{|\partial\mathfrak{n}(\gamma_{r,i},E_3)/\partial \gamma_{r,i}|}\le C\frac{N^{\tau}}{N}(\frac{i}{N})^{-1/3}=C N^{\tau}N^{-2/3}i^{-1/3}
   \end{equation*}
   with high probability. Then, the result for Theorem \ref{thm: Rigidity} follows by carefully figuring out the order of $i$ in each bulk component.

For the case that $\phi\ll1$, one should notice that the typical order of the eigenvalues of $W$ is of $\mathrm{O}(\phi^{-1/2})$. Then, one should rescale the integral involving $\rho_W$ and $\rho$ as 
\begin{equation*}
    \int_a^b\rho_W(\mathrm{d}\lambda)=\int_{\phi^{1/2}a}^{\phi^{1/2}b}\phi^{-1/2}\rho_W(\mathrm{d}(\phi^{1/2}\lambda)),
\end{equation*}
which results in the scaling factor in $\phi^{-1/2}\rho_W$ being $M^{-1}$.

\subsection{Proof of Theorem \ref{thm: Local law outside the spectrum}}
From eigenvalue rigidity (Theorem \ref{thm: Rigidity}), we have 
\begin{equation}
    |\lambda_{k,i}-\gamma_{k,i}|\prec K^{-2/3}\cdot i^{-1/3},
\end{equation}
for some $i=1,\dots,N_l$ satisfying $R-c\le \gamma_{l,i}$. With the convention $\gamma_{1,0}=R$, we may write
\begin{equation}
    m_{W}(z)-m_1(z)=\sum_{k=1}^p\sum_{i=1}^{N_l}\int_{\gamma_{k,i}}^{\gamma_{k,i-1}}\rho(\mathrm{d}x)\Big(\frac{1}{\lambda_i-z}-\frac{1}{x-z}\Big).
\end{equation}
We find that for $x\in[\gamma_{k,i},\gamma_{k,i-1}]$
\begin{equation}
    |\lambda_{k,i}-\gamma_{k,i}|+|x-\gamma_{k,i}|\prec K^{-2/3}\cdot i^{-1/3}
\end{equation}
with high probability. Since $z\in\mathbf{D}_{os}$, $|z-\gamma_{k,i}|\ge K^{-2/3+\delta}$ for all $i$. Besides, $\epsilon$ can be made sufficiently small, so we have 
\begin{equation*}
    |m_{W}(z)-m_1(z)|\prec \frac{1}{K}\sum_{i=1}^KK^{-2/3}\cdot i^{-1/3}\frac{1}{|\gamma_{k,i}-z|^2}.
\end{equation*}

By the definitions of $\gamma_{k,i}$'s and the square root decay for $\rho$ near the edge, we obtain that $|R-\gamma_{k,i}|\asymp (i/K)^{2/3}$ for $i\le N_l$. Then by $\sqrt{\kappa^2+\eta^2}\asymp\kappa+\eta$, $\kappa^2+\eta^2\asymp (\kappa+\eta)^2$, we conclude that 
\begin{equation*}
    |m_{W}(z)-m_1(z)|\prec \frac{1}{K}\frac{1}{(\kappa+\eta)+(\kappa+\eta)^2}.
\end{equation*}
The proof of Theorem \ref{thm: Local law outside the spectrum} is finished.

\section{Application to spiked covariance model}\label{sec: Application}

\subsection{Estimation of spiked eigenvalues}
The spiked covariance model, originally introduced by \cite{Johnstone01}, assumes that the spectrum of the covariance matrix $\Sigma$ forms several separate groups, i.e.,
\begin{align}\label{spike-model}
 {\rm Spec}(\Sigma)=\bigg(\underbrace{\alpha_{1},\ldots,\alpha_{1}}_{q_1}, \ldots, \underbrace{\alpha_{L},\ldots,\alpha_{L}}_{q_L},\ \underbrace{\beta_{1},\ldots,\beta_{M-\mathcal L}}_{M-\mathcal L}\bigg),\quad \mathcal L=\sum_{\ell=1}^L q_\ell.
\end{align}
In this spectrum, the eigenvalues $\alpha_1 > \alpha_2 > \cdots > \alpha_L$ are referred to as {\em spiked eigenvalues} with multiplicities $\{q_1,\ldots,q_L\}$, and the remaining ones $\{\beta_1,\ldots, \beta_{M-\mathcal L}\}$ are called {\em bulk eigenvalues}. To simplify the notation, we assume that the spikes are larger than the bulk eigenvalues, which can be easily extended to general situations.
%

Estimating the spiked eigenvalues $\{\alpha_1,\ldots,\alpha_L\}$ is one of the central inferential tasks for this model, which requires a spectrum separation condition described below.

{\bf Assumption (A4).}\label{ass_app_spikemodel} For the spiked covariance model \eqref{spike-model}, we assume that
$$
\liminf_N \min_{1\leq s\neq t\leq L} |\alpha_s-\alpha_t|>0\quad \text{and}\quad 
\liminf_N \min_{1\leq \ell\leq L} \psi_{\phi}^{\prime}(\alpha_\ell)> 0,
$$
where
$$
\psi_{\phi}(x)=\frac1{\sqrt{\phi}}x+\sqrt{\phi} x \int \frac{t}{x-t} d \pi_{L+1}(t),\ \pi_{L+1}=\frac{1}{M}\sum_{i=1}^{M-\mathcal L}\delta_{\beta_i} \text{ and }  x\notin {\rm Supp}(\pi_{L+1}).
$$ 
This assumption states that the spikes $\{\alpha_1,\ldots,\alpha_L\}$ must be distinguishable and be distant from the bulk eigenvalues, which we refer to as {\em distance spikes}.
In this context, the $L$ spikes of $\Sigma$ give rise to $L$ separate clusters of sample eigenvalues, forming a one-to-one correspondence, and they are also isolated from the bulk sample eigenvalues. See \cite{BSbook} for more details. 

 A small simulation is carried out to illustrate the spectral separation for $L=1$. We set the spectrum of $\Sigma$ to be Spec$(\Sigma)=\{4,1,\ldots,1\}$. The dimensional settings are $(M,N)=(400,40000)$, $(400,400)$, and $(40000,400)$, representing $\phi$ approaching $0$, a positive constant $\phi_\infty=1$, and $\infty$, respectively.
Notice that the derivative $\psi'_\phi(x)$ is 
$$
\psi_{\phi}^{\prime}(x)=\frac1{\sqrt{\phi}}-\frac{\sqrt{\phi}}{(x-1)^2}
,\quad \text{and thus}\quad \liminf_N\psi_{\phi}^{\prime}(4)>0,~ \text{if}~ \phi\to\phi_\infty \in [0,9).
$$
As $\phi$ approaches a limit within the range of $[0,9)$, a sample spike appears outside the bulk ones, as shown in (a) and (b) of Figure \ref{figmn}. However, as $\phi$ increases beyond 9, all the sample eigenvalues combine without a spike, as seen in (c) of Figure \ref{figmn}.

\begin{figure}[htbp]
\centering
\begin{minipage}{0.32\linewidth}
\vspace{0.4pt}
\includegraphics[width=\textwidth]{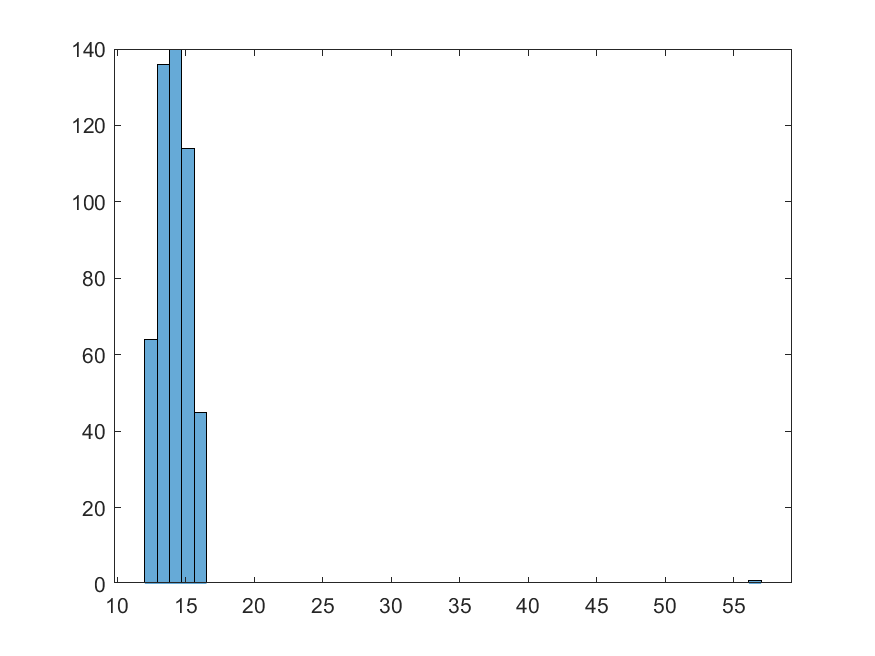}
\centerline{(a)}
\end{minipage}
\begin{minipage}{0.32\linewidth}
\vspace{0.4pt}
\includegraphics[width=\textwidth]{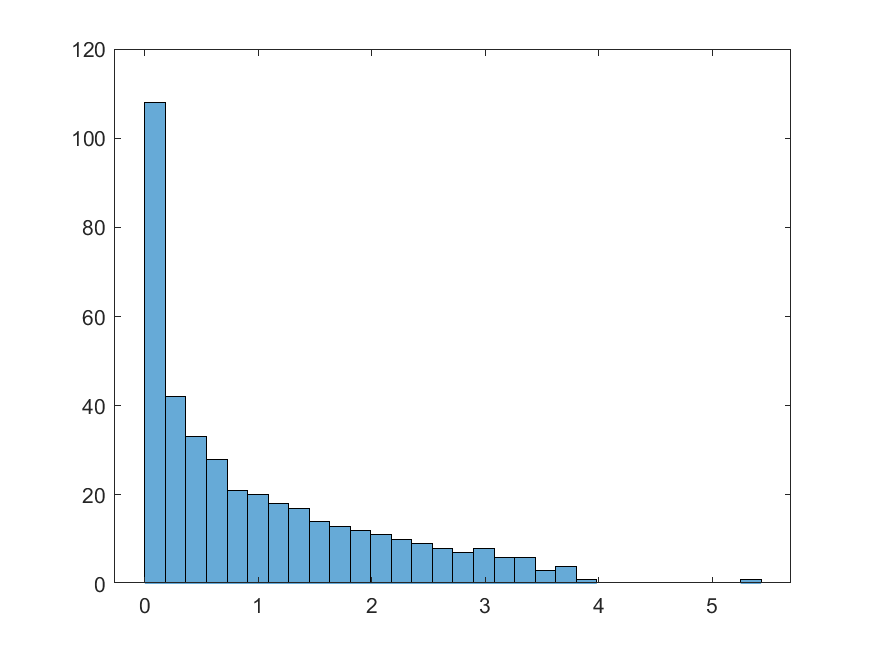}
\centerline{(b)}
\end{minipage}
\begin{minipage}{0.32\linewidth}
\vspace{0.4pt}
\includegraphics[width=\textwidth]{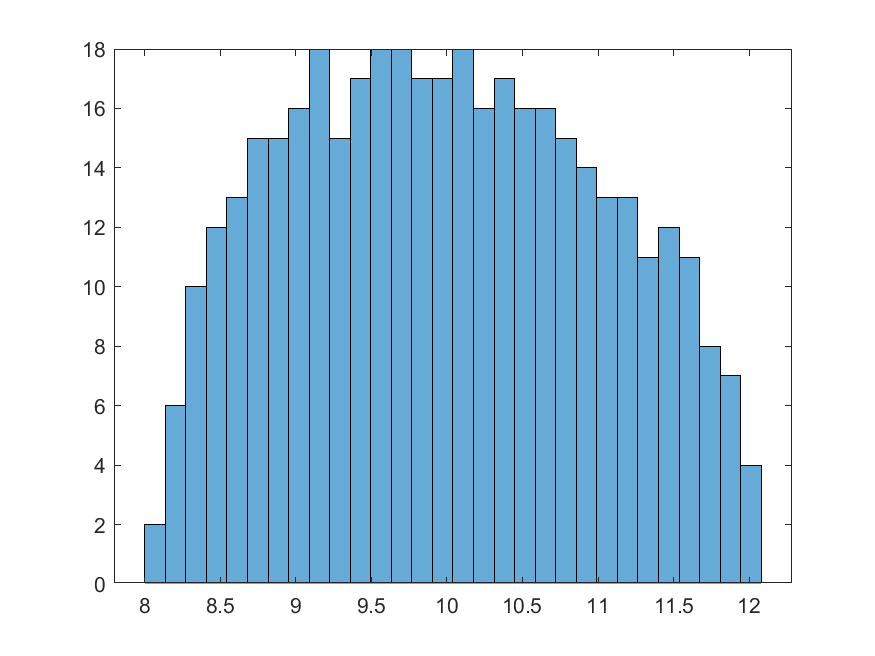}
\centerline{(c)}
\end{minipage}
\caption{Eigenvalues of the sample covariance matrix with dimensions
 $(M,N)$: (a) $(400,40000)$,  (b) $(400,400)$, and (c) $(40000,400)$.}
\label{figmn}
\end{figure}

In the work of \cite{baiandding2012}, the authors introduced a method to estimate the spikes $\alpha_1, \ldots, \alpha_L$, relying on the following identity:
\begin{align*}
\alpha_\ell
=-\frac{\sqrt{\phi}}{m_1[\psi_{\phi}(\alpha_\ell)]}\quad \text{for}~ \ell=1,\ldots,L.
\end{align*}
By substituting the empirical counterparts of $m_1(z)$ and $\psi_{\phi}(\alpha_\ell)$, the authors proposed a set of estimates $\{\hat\alpha_{B}(\lambda_j), \lambda_j\in \mathcal R_{\ell}\}$ for $\alpha_\ell$, defined as
	\begin{align*}
	\hat\alpha_{B}(\lambda_j) 
	=-\sqrt{\phi}\left[\frac{1}{N} \sum_{\lambda_k \notin \mathcal R_\ell} \frac{1}{\lambda_k-\lambda}_j\right]^{-1}, \quad \lambda_j\in\mathcal R_\ell,	
	\end{align*}
where $\{\lambda_j, j=1,\ldots,N\}$ are the $N$ eigenvalues of the matrix $W$, sorted in descending order,
$\mathcal R_{\ell}$ denotes the $\ell$th cluster of sample spiked eigenvalues, and $\lambda_j$ can be any spikes within this cluster.
Nevertheless, there is currently no criterion for selecting $\lambda_j$'s. To address this concern, we observe that the random vector
$$
\left\{\sqrt{N}(\hat\alpha_{B}(\lambda_j)-\alpha): \lambda_j\in\mathcal R_\ell\right \}
$$ 
will converge in distribution to the eigenvalues of a zero-mean Gaussian matrix under certain conditions, see \cite{bai2008central}. This suggests that an individual estimate $\hat\alpha_{B}(\lambda_j)$ may be potentially biased, but the average of these estimates can avoid such an issue. We thus consider an averaged estimate of $\alpha_\ell$, i.e., 
$$
\hat\alpha_{B,\ell}=\frac1{q_\ell}\sum_{\lambda_j\in \mathcal R_\ell} \hat\alpha_{B}(\lambda_j).
$$
Using local laws, we can establish the convergence rate for the estimate $\hat\alpha_{B,\ell}$ in general asymptotic regimes. 
\begin{theorem}\label{app}
Suppose that Assumptions (A1)-(A2)-(A3)-(A4) hold. Then, we have
\begin{align*}
|\hat\alpha_{B,\ell} -\alpha_\ell| \prec N^{-\frac12}
\end{align*}
for $\ell=1,\ldots,L$.
\end{theorem}

There is an alternative estimate of $\alpha_\ell$ introduced in \cite{mestre2008improved}, which is formulated using contour intergation as 
\begin{align}\label{Mestre}
\hat\alpha_{M,\ell} 
& = \frac{-N\sqrt{\phi}}{2 \pi \mathrm{i}~ q_\ell} \oint_{\mathcal C_\ell} \frac{z}{m_W(z)} d m_W(z)
=\frac {N\sqrt{\phi}}{q_\ell} \sum_{\lambda_\ell\in \mathcal R_\ell} \left(\lambda_\ell-\mu_\ell\right),\quad \ell=1,\ldots, L.
\end{align}
In this formula, $\mathcal C_\ell$ is a simple contour, counterclockwise orientated, enclosing only the cluster $\mathcal R_\ell$ of sample spikes, and $\{\mu_\ell:\mu_\ell\in (\lambda_{\ell+1}, \lambda_\ell),~ m_W(\mu_\ell)=0\}$  is a set of poles contained inside the contour. One limitation of this estimation lies in the absence of a theoretical guarantee of consistency when $q_{\ell}=\mathrm{O}(1)$.

When the spike $\alpha_\ell$ is simple, i.e., $q_\ell=1$, Mestre's estimation is closely related to Bai-Ding's method. In particular, we have
\begin{align*}
m_W(\mu_\ell)=\frac1N\frac{1}{\lambda_\ell-\mu_\ell}+\frac1N \sum_{k\neq \ell}\frac{1}{\lambda_k-\mu_\ell}=0
\end{align*}
and $\lambda_\ell-\mu_\ell=o_{a.s.}(1)$, which give
\begin{align*}
\hat\alpha_{M,\ell} = \hat\alpha_{B,\ell}+ \mathrm{O}_{a.s.}\left(\sqrt{\phi}(\lambda_\ell-\mu_\ell)\right).	
\end{align*}
Therefore, the two estimates have the same convergence rate from Theorem \ref{app}.

\subsection{Proof of Theorem \ref{app}.}

Let $\Sigma^0$ denote a new covariance matrix obtained by removing all the spiked eigenvalues from $\Sigma$. We define $W=X^{*}\Sigma X$ and $W^0=X^{*}\Sigma^0X$, along with their ESDs denoted by $\rho_W$ and $\rho_{W^0}$, and their Stieltjes transforms denoted by $m_W$ and $m_{W^0}$.

We prove the theorem by showing two lemmas.
The first lemma gives the eigenvalue sticking in the bulk eigenvalues of $W$ and $W^0$. The second one shows the exact location for the spiked eigenvalues of $W$. 
\begin{lemma}
    For the spiked covariance model $\Sigma$ in \eqref{spike-model} and $\Sigma^0$ described above. Suppose $\mathcal{L}$ is a finite number. Then
    \begin{equation}
        |m_W(z)-m_{W^0}(z)|=\mathrm{O}(\frac{1}{N\eta}).
    \end{equation}
\end{lemma}
\begin{proof}
By the Cauchy interlacing property, we know that
    \begin{equation*}
        \lambda_{\mathcal{L}+k}(W)\le \lambda_{k}(W^{0})\le\lambda_{k}(W),\quad k=1,\cdots, N-\mathcal{L}.
    \end{equation*}
    This implies that 
    \begin{equation*}
        \sup_{x\in\mathcal{R}}|\rho_{W}(x)-\rho_{W^0}(x)|\le \frac{\mathcal{L}}{N}
    \end{equation*}
and therefore, 
\begin{equation*}
        |m_W(z)-m_{W^0}(z)|\le \int\frac{1}{|x-z|}\mathrm{d}(\rho_{W}(x)-\rho_{W^0}(x))\le \eta^{-1}N^{-1}\int\frac{\mathcal{L}\eta}{(x-E)^2+\eta^2}\mathrm{d}x\le\frac{\pi\mathcal{L}}{N\eta}.
    \end{equation*}
\end{proof}

\begin{lemma}
    Suppose that Assumption (A4) holds for the spiked covariance model in \eqref{spike-model}. Then, we have
    \begin{equation}
        |\lambda_{j}-\psi_{\phi}(\alpha_{\ell})|\prec K^{-1/2},\quad \forall \lambda_j\in \mathcal R_\ell,
    \end{equation}
    for $\ell=1,\dots,L$.
\end{lemma}

\begin{proof}
    The proof of this lemma is a perturbative one and similar to those in \cite[Theorem 2.3]{bloemendal2016principal} especially for the case $\phi\gtrsim1$. In the sequel, we only investigate the case where $K=1$, $q_1=1$ with $\phi\ll1$. We give the estimation for the first eigenvalue $\lambda_{1,1}$ of $W$, while the other cases can be handled similarly. Denote the interval $I_0:=[\psi_{\phi}(\alpha_{1,1})-M^{-1/2+\epsilon},\psi_{\phi}(\alpha_{1,1})+M^{-1/2+\epsilon}]$. Then, it is easy to see that if $\lambda$ is an eigenvalue of $W$ lying in $I_0$, then $\lambda$ can not be the eigenvalue of $W^0$. Consequently, we have 
    \begin{gather}
        \operatorname{det}(\lambda I-X^{*}(\Sigma^0+\alpha_{1,1}\mathbf{u}_1\mathbf{u}_1^{*})X)=0\\
        \Rightarrow1+\alpha_{1,1}\mathbf{u}_1^{*}X(X^{*}\Sigma^0X-\lambda I)^{-1}X^{*}\mathbf{u}_1=0,
    \end{gather}
    where $\mathbf{u}_1$ is the corresponding eigenvector of $\Sigma$ corresponding to $\alpha_{1,1}$. Since $\mathbf{u}_1$ is orthogonal to the eigenspace of $\Sigma^0$, one has from large deviation bounds for the elements in $X$ that 
    \begin{equation*}
        \mathbf{u}_1^{*}X(X^{*}\Sigma^0X-\lambda I)^{-1}X^{*}\mathbf{u}_1-\frac{1}{\sqrt{MN}}\operatorname{tr}(X^{*}\Sigma^0X-\lambda I)^{-1}\prec \frac{1}{\sqrt{MN}}\sqrt{\frac{\operatorname{Im}\operatorname{tr}G^0(\lambda)}{\eta}},
    \end{equation*}
    where $G^0(z):=(X^{*}\Sigma^0X-z I)$. On the other hand, we have from Theorem \ref{thm: locallaw_phi>1}
    \begin{equation*}
            |\frac{1}{\sqrt{MN}}\operatorname{tr}G^0(\lambda)-\phi^{-1/2}m_1^0(\lambda)|\prec \frac{\phi^{-1/2}}{K(\kappa+\eta)^2}.
    \end{equation*}
    Recall that for $\phi\gg1$, one has $\operatorname{Im}m_1^0(z)\sim\eta$ for $z$ outside spectrum with $(\kappa+\eta)\sim1$. On the other hand, one should notice that for $z$ around locates the spiked eigenvalue (outside spectrum), we have 
    \begin{equation*}
        \operatorname{Im}m_1^0(z)=\operatorname{Im}\int\frac{1}{x-z}\rho^0(\mathrm{d}x)=\int\frac{\eta}{(x-E+\eta)^2}\rho^0(\mathrm{d}x)\sim \phi\eta,
    \end{equation*}
    where $\rho^0$ is the limiting density of $\rho_{W^0}$ and we also used the fact that $(\kappa+\eta)\sim\phi^{-1/2}$. It gives that for $\phi\ll1$,
    \begin{gather*}
        1+\alpha_{1,1}\phi^{-1/2}m_1^0(\lambda)+\mathrm{O}_{\prec}(\frac{\phi^{1/2}}{\sqrt{M}})=0\\
        \Rightarrow  \frac{1}{m_1^0(\lambda)}+\phi^{-1/2}\alpha_{1,1}=\mathrm{O}_{\prec}(\frac{1}{\sqrt{M}}),
    \end{gather*}
    where we used the fact that $|m_1^0(\lambda)|\sim\phi^{1/2}$. Besides, recall that $m_1^0(\lambda)$ satisfies
    \begin{equation*}
        \frac{1}{m_1^0(\lambda)}=-\lambda+\phi^{1/2}\int\frac{x}{1+\phi^{-1/2}m_1^0(\lambda)x}\pi^0(\mathrm{d}x).
    \end{equation*}
    Consequently, we may observe that 
    \begin{equation}
        \lambda=\phi^{-1/2}\alpha_{1,1}+\phi^{1/2}\alpha_{1,1}\int\frac{x}{\alpha_{1,1}-x}\pi^0(\mathrm{d}x)+\mathrm{O}_{\prec}(\frac{1}{\sqrt{M}}).
    \end{equation}
    It implies that $|\lambda-\psi_{\phi}(\alpha_{1,1})|\prec M^{-1/2}$. This completes the proof.
\end{proof}

\bibliographystyle{abbrvnat} 
\bibliography{ref}       





\end{document}